\documentclass[10pt]{article}
\usepackage{amsfonts}
\usepackage{amssymb}
\usepackage{amsthm}
\usepackage{amsmath}
\usepackage{bbm}
\usepackage{graphicx}
\usepackage{latexsym}
\usepackage{mathrsfs}
\usepackage[all]{xy}

\usepackage{geometry}
\newcommand{\cA}{\mathcal{A}}
\newcommand{\cB}{\mathcal{B}}
\newcommand{\cC}{\mathcal{C}}

\newcommand{\cG}{\mathcal{G}}
\newcommand{\cH}{\mathcal{H}}
\newcommand{\cI}{\mathcal{I}}

\newcommand{\cL}{\mathcal{L}}
\newcommand{\cM}{\mathcal{M}}

\newcommand{\cO}{\mathcal{O}}

\newcommand{\cS}{\mathcal{S}}

\newcommand{\cU}{\mathcal{U}}

\newcommand{\rA}{\mathrm{A}}
\newcommand{\rB}{\mathrm{B}}
\newcommand{\rC}{\mathrm{C}}

\newcommand{\rW}{\mathrm{W}}

\newcommand{\rf}{\mathrm{f}}

\newcommand{\rh}{\mathrm{h}}

\newcommand{\rp}{\mathrm{p}}

\newcommand{\ru}{\mathrm{u}}

\newcommand{\Si}{\Sigma} 

\newcommand{\bC}{\mathbb{C}}
\newcommand{\bT}{\mathbb{T}}
\newcommand{\bM}{\mathbb{M}}
\newcommand{\bN}{\mathbb{N}}
\newcommand{\bU}{\mathbb{U}}
\newcommand{\bZ}{\mathbb{Z}}

\newcommand{\si}{\sigma}

\newcommand{\e}{\epsilon}
\newcommand{\io}{\iota}

\newcommand {\norm}[1]{\Vert{#1}\Vert}    

\newcommand{\ad}{\mathrm{ad}}

\author{\textsc{Giuseppe Ruzzi$^{1}$ and Ezio Vasselli$^{2}$}\footnote{Both the authors are supported by the  EU network ``Noncommutative Geometry" MRTN-CT-2006-0031962.}\\
  \null\\
\small{$^{1}$Dipartimento di Matematica, Universit\`a di Roma ``Tor Vergata'',}\\
\small{Via della Ricerca Scientifica, I-00133 Roma,  Italy.}  \\
\small{\texttt{ruzzi@mat.uniroma2.it}} \\[5pt]
\small{$^{2}$Dipartimento di Matematica, Universit\`a di Roma ``La Sapienza'',}\\
\small{Piazzale Aldo Moro 5, I-00185 Roma, Italy.}\\
 \small{\texttt{ vasselli@mat.uniroma2.it  }}\\[20pt]
{\bf\small{Dedicated to  John E. Roberts on the occasion of his   
            seventieth  birthday}}}
        
\date{}

\title{\textsc{A new light on nets of C*-algebras
        and their representations}} 

\begin{document}
\maketitle

\begin{abstract}
The present paper deals with the question of representability of nets of $\rC^*$-algebras 
whose underlying poset, indexing the net, is not upward directed. A particular class of nets, 
called $\rC^*$-net bundles, is classified in terms of $\rC^*$-dynamical systems having as group  
the fundamental group of the poset. Any net of $\rC^*$-algebras 
embeds into a unique $\rC^*$-net bundle, the enveloping net bundle, which generalizes 
the notion of universal $\rC^*$-algebra given by Fredenhagen to nonsimply connected posets.
This allows a classification of nets; in particular, we call injective those nets having a 
faithful embedding into the enveloping net bundle. Injectivity
turns out to be equivalent to the existence of faithful representations. We further 
relate injectivity  to a generalized \v Cech cocycle of the net, and this allows 
us to give examples of nets exhausting the above classification. 
Using the results of this paper we shall show, in a forthcoming paper,  that any conformal net over $S^1$ is injective. 
\end{abstract}

\newpage 

\tableofcontents
\markboth{Contents}{Contents}

\newpage 


  \theoremstyle{plain}
  \newtheorem{definition}{Definition}[section]
  \newtheorem{theorem}[definition]{Theorem}
  \newtheorem{proposition}[definition]{Proposition}
  \newtheorem{corollary}[definition]{Corollary}
  \newtheorem{lemma}[definition]{Lemma}

  \theoremstyle{definition}
  \newtheorem{remark}[definition]{Remark}
    \newtheorem{example}[definition]{Example}

\theoremstyle{definition}
  \newtheorem{ass}{\underline{\textit{Assumption}}}[section]


\numberwithin{equation}{section}

\section{Introduction}

A net of $\rC^*$-algebras is a covariant functor 
from a poset, considered as a category, to the category of unital $\rC^*$-algebras having 
faithful ${^*}$-morphisms as arrows. Actually, this structure is not a net 
unless the poset is upward directed; it is, rather, a precosheaf of 
$\rC^*$-algebras, however we prefer to maintain the term 
net throughout this paper, according to the convention used in algebraic 
quantum field theory. The present paper is  addresses 
to the analysis of nets over nonupward directed posets, and, in particular,
the question of their representability on Hilbert spaces.\\  
\indent The basic idea  of  the algebraic approach to 
quantum fields over a spacetime 
is that the physical content of the theory is completely encoded in the relation linking 
observables measurable within a suitable region of the spacetime to that region 
(\cite{HK, Haa, Ara, Rob, BH, BrF}). 
Mathematically this is expressed by  
a net of $\rC^*$-algebras defined over a poset given as a suitable set of regions 
of the spacetime ordered under inclusion. 
Two physical inputs are imposed on this net: causality and covariance under
spacetime symmetries, if there are any.   This is  what is called  the \emph{observable net}.
Quantum physical systems, like for instance  elementary particles, are described  by certain  representations 
of the observable net on a Hilbert space. Some remarkable results 
in this direction are \cite{DHR1,DHR2,BF}.\\ 
\indent The set of regions indexing the observable net must be chosen to best fit the topological and the 
causal properties of the spacetime and the global symmetries.  
In  Minkowski space this is  the set  
of double cones, which turns out to be upward directed under inclusion; so the net embeds faithfully in the 
inductive limit $\rC^*$-algebra (the colimit). Symmetries on the net lift to the inductive limit, and 
(covariant) representations of the inductive limit yield 
(covariant) representations of the observable net.  
However, when one deals with theories over curved spacetimes, where nontrivial topologies are allowed, 
or over the circle $S^1$, 
the appropriate set of regions is not upward directed any more
\footnote{An upward directed poset is simply connected; 
the first homotopy group of the spacetime is isomorphic to that of the poset indexing the net,
see \cite{Ruz}.}, 
see \cite{BMT,GLRV,Ruz}.
So the question of the existence of representations arises. \\
\indent A similar problem arises in geometric group theory  (\cite{BHa,GHV}). 
Any simple complex of groups is a net of groups over its set of  
simplices ordered under opposite inclusion, 
and the realizability of the complex,
namely the existence of a faithful embedding into the colimit, is equivalent to the existence of 
a faithful representation of the net (the analogue of a Hilbert space representation, see below). 
Neverthless the powerful results obtained in this context seem to be very far from 
to be applicable to the nets arising in quantum field theory, mainly because 
the posets involved are finite and, in many applications, the groups are 
finitely generated. \\
\indent The question of representability in the context of quantum field theory
has been considered by Fredenhagen, who generalized the inductive limit $\rC^*$-algebra to nonupward directed posets 
introducing the universal $\rC^*$-algebra of a net, characterized by 
the property that representations of the net (those that we call, in the present paper, 
\emph{Hilbert space representations}) lift to representations of the universal $\rC^*$-algebra (\cite{Fre,GL}). 
Symmetries of the net lift to the universal $\rC^*$-algebra, which reduces to the inductive limit 
when the poset is upward directed. The advantage, apart that of dealing with 
a single algebra rather than a net,  is that some of the symmetries and representations 
of the observable net may be realized in the universal $\rC^*$-algebra, see for instance \cite{FRS,DFK,Hol}. \smallskip

Nevertheless we consider this approach unsatisfactory, for two main reasons. 
\emph{First}, the universal $\rC^*$-algebra may be trivial, but there is no 
result connecting the nontriviality of the universal $\rC^*$-algebra to
intrinsic properties of the net itself. \emph{Secondly}, the universal 
$\rC^*$-algebra is too restrictive, since it does not take  the fundamental group of the poset
into account\footnote{A geometrical invariant appearing in this context is the 2-homology 
of the poset. As shown in \cite{Hol},
central elements of the universal algebra of a net are associated 
with 2-cycles of the spacetime.}.  
Two recent papers,
\cite{CKL} and \cite{BR} (see also \cite{BFM}) illustrate the problems that arise.
In particular, in \cite{BR} quantum effects 
due to the topology of the spacetime have been studied.
Guided  by Roberts' cohomology \cite{Rob1,Rob}, the authors introduced a different notion  
of representation for a net of $\rC^*$-algebras, that we use here and call 
\emph{representation} in the sequel.   
These induce a representation of the fundamental group of the poset; generalize the above notion of Hilbert
space representation  in the sense that the two notions coincide when the poset is simply connected; 
describe, as representations of 
the observable net over a spacetime, charges affected by the topology of the spacetime. 
Nevertheless, these representations do not, in general,  admit
any extension to the universal $\rC^*$-algebra.\smallskip

The present paper answers  the above two questions and is organized
as follows. 
We begin by giving a combinatorial construction of nets of $\rC^*$-algebras 
over arbitrary posets (\S \ref{Bb}). 
Afterwards we focus on a particular 
class of nets, the  $\rC^*$-net bundles, giving a classification 
in terms of $\rC^*$-dynamical systems carrying an action of the fundamental group 
of the poset (\S \ref{Bc}). 
The importance of net bundles is  that any 
net embeds into a $\rC^*$-net bundle called \emph{the enveloping net bundle}, 
uniquely associated with the given net.  
This gives a first answer to the above questions, since 
it  takes into account the topology of the poset and reduces to the universal $\rC^*$-algebra 
when the poset is simply connected. We distinguish nets between {\em nondegenerate},
whenever the enveloping net bundle is nonvanishing,
and \emph{injective}, whenever the embedding is faithful (\S \ref{Bd}). 
Injectivity turns out to be equivalent to the existence of faithful representations (\S \ref{C}).
We give an intrinsic description of injective nets in terms of a generalized \v Cech cocycle. 
This allows us to find  examples of nets exhausting the above classification: 
degenerate nets, nondegenerate but noninjective nets and  injective nets having no Hilbert space representations
(so, the universal $\rC^*$-algebra is trivial but the embedding 
into the  enveloping net bundle is faithful, \S \ref{D}). \smallskip

Finally, we stress that this paper shall be followed by a second one (\cite{RVs1}),
in which the ideas of the present work shall be used to prove that any (covariant) 
net over the standard, nondirected subbase of $S^1$ has faithful (covariant) representations,
a scenario of interest in conformal field theory.
\smallskip

\section{Preliminaries}
\label{A}

We introduce basic notions of posets, some  motivated by 
algebraic quantum field theory, and some related algebraic structures. 
In particular we discuss the fundamental group of a poset in terms of a 
related simplicial set. 


\subsection{Posets} 
\label{A:a}

A \emph{poset} is a nonempty set $K$ with an order relation $\leq$, that is, 
$\leq$ is  a binary relation which is reflexive, antisymmetric and transitive. 
We shall denote the elements of $K$ by Latin letters $o,a$. We shall write 
$a<o$ to indicate that $a\leq o$ and $a\ne o$.
A poset $K$ is said to be \emph{upward directed} if for any pair  $o_1,o_2\in K$ 
there is $o\in K$ with $o_1,o_2\leq o$. 
The \emph{dual} poset of $K$ is the set $K^\circ$ having the same elements
as $K$ and order relation $\leq^\circ$ defined by 
$ a\leq^\circ o$ if, and only if,  $o\leq a$. 
We say that $K$ is \emph{downward directed} if $K^\circ$ is upward directed.
A subset $C \subseteq K$ is said to be \emph{contained} in $a \in K$
whenever $o \leq a$ for all $o \in C$ and in this case we write $C \subseteq a$.\\
\indent A \emph{morphism} from $K$ to a poset $P$ is an order preserving map $\rf:K\to P$: 
$\rf(o)\leq \rf(a)$ whenever 
$o\leq a$, and we say that it is an isomorphism if it is injective and surjective. 
We shall denote the set of automorphisms of a poset $K$ by $\mathrm{Aut}(K)$. 
A group $G$ is a \emph{symmetry group} for $K$ if there is an order preserving  left  
action $G\times K\ni (g,o)\to go \in K$, that is, if there is an injective group morphism 
from $G$ to $\mathrm{Aut}(K)$.\\
\indent The next definition is a key notion for posets used in algebraic quantum field theory. 
A \emph{causal disjointness relation} for a poset 
$K$ is symmetric binary relation 
$\perp$ such that 
\[
 a\perp o \ \ {\mathrm{and}} \ \ \tilde o\leq o \ \ \Rightarrow \ \ a\perp \tilde o \ .  
\]
If $K$ is endowed with  a causal disjointness relation $\perp$ and $G$ is a symmetry group 
for $K$, \emph{we always assume that $G$ preserves $\perp$}; 
this amounts to saying that 
\[
 a\perp o \ \ \ \iff \ \ \ ga \perp go \ \qquad , 
\]
for any $g\in G$.\\ 
\indent In algebraic quantum field theory, the main object of study is the observable net: 
an inclusion preserving mapping from a set of regions $K$ of a given spacetime manifold 
to the set of $\rC^*$-algebras \cite{Haa}.  So the poset structure of the set $K$, ordered under  inclusion, enters the theory. In general this set of regions $K$ is a base of neighbourhoods for the topology of the spacetime manifold, 
consisting of open, connected and simply connected subsets of the spacetime. If the spacetime has a 
global symmetry group, then one considers only bases of neighbourhoods
stable under its action. \\
\indent In Minkowski space $\mathbb{M}^4$, $K$ is the set of double cones, the symmetry group is the Poincar\'e group, and the causal disjointness relation is  spacelike separation.
This poset is upward directed under inclusion. For arbitrary 
4-dimensional globally hyperbolic spacetimes $\cM$, $K$ is the set 
of diamonds \cite{Ruz,BR}, which is stable under isometries, and the causal disjointness
relation is induced by the causal structure of the spacetime. 
This poset is not upward directed when $\cM$ is not simply connected or when 
$\cM$ has compact Cauchy surfaces. 
For theories on the circle $S^1$, $K$ is the set of connected open intervals of $S^1$ having a 
proper closure; the symmetry group is $\mathrm{Diff}(S^1)$ or the M\"obius subgroup; 
the usual notion of disjointness between sets is the causal disjointness relation. 
This poset is not upward directed either.  
%
%
\subsection{A simplicial set for  posets}
\label{A:b}
We introduce a simplicial set associated with a poset and, 
 in particular, discuss the notion of the fundamental group of a poset in terms of this 
simplicial set. The standard symbols $\partial_i$ and $\si_i$ are used to denote the face and degeneracy maps. The symbols $\partial_{ij}$ and $\si_{ij}$ denote, respectively, 
the compositions $\partial_{i}\partial_j$ and $\si_i\si_j$.  
References for this section are \cite{Rob,RRV}. \bigskip

We consider the simplicial set $\Si_*(K)$ of \emph{singular simplices}
associated with a poset $K$, introduced by Roberts in \cite{Rob}.  
A brief description of the set $\Si_n(K)$ of $n$-simplices
is the following.  A $0$-simplex is just an element of $K$. Inductively, for $n\geq 1$, 
and $n$-simplex $x$  is formed by $n+1$, $(n-1)$-simplices
$\partial_0x,\ldots, \partial_nx$ and by an element of the poset 
$|x|$, called the \emph{support} of $x$, such that 
$|\partial_ix|\leq |x|$ for $i=0,\ldots,n$. We shall denote 
0-simplices either by $a$ or by $o$, 1-simplices by $b$, and
2-simplices by $c$. Given a 1-simplex $b$ the \emph{opposite} 
$\overline{b}$ is the 1-simplex having the same support as $b$ 
and such that $\partial_0\overline{b}=\partial_1b$,
$\partial_1\overline{b}=\partial_0b$.\\
\indent Composing 1-simplices one gets paths. 
A \emph{path} $p$ is an expression  of the form  $b_n*\cdots *b_1$
where $b_i$ are 1-simplices satisfying the relations $\partial_0b_{i-1}= \partial_1b_{i}$ 
for $i=2,\ldots,n$. We define the 0-simplices  
$\partial_1p:=\partial_1b_1$ and
$\partial_0p:=\partial_0b_n$ and call them, respectively, 
the starting and the ending point of $p$. The \emph{support} of a path $p$ is 
the subset $|p|$ of $K$ whose elements are the  the supports of the 1-simplices 
which compose the path. 
By $p:a\to \tilde a$ we mean a path starting from $a$
and ending at $\tilde a$. A path $p:o\to o$ is called a \emph{loop} 
over $o$. The \emph{opposite} of $p$ 
is the path $\overline{p}:\tilde a\to a$ defined by 
$\overline{p}:=\overline{b}_1*\cdots *\overline{b}_n$. 
If $q$ is a path from $\tilde a$ to $\hat
a$, then we can define, in an obvious way, the composition 
$q*p:a\to \hat a$. \\
\indent Any poset morphism $\rf:K\to P$ induces a morphism between the corresponding 
simplicial sets. 
Given $o\in \Si_0(K)$, we let $\rf(o)\in \Si_0(P)$ be the image of $o$ by $\rf$.
Inductively, for $n\geq 1$, given an $n$-simplex $x$ of $K$ we define $\rf(x)$ 
as the $n$-simplex of $P$ withs  faces 
$\partial_i\rf(x):= \rf(\partial_ix)$ for $i=0,1,\ldots n$, and support  
$|\rf(x)|:=\rf(|x|)$. This, clearly, induces a mapping between 
the corresponding set of paths: $\rf(p):= \rf(b_n)*\cdots*\rf(b_2)*\rf(b_1)$ 
is a path of $P$ for any path $p$ of $K$ of the form $p=b_n*\cdots *b_2*b_1$. \\
\indent A poset $K$ is said to be connected whenever for any pair $o,a$ of $0$-simplices 
there is a path $p:o\to a$. In the present paper \emph{we shall always consider pathwise 
connected posets}. In a pathwise connected poset we can define \emph{path frames}: 
fixed a 0-simplex $o$, the pole, a path frame $P_o$, with respect to $o$, is a choice 
for any 0-simplex $a$ of a  path $p_{(a,o)}:o\to a$ such that $p_{(o,o)}$ is homotopic (see below) 
to $\iota_o$ the trivial loop over $o$ which i.e. the  degenerate 1-simplex $\si_0o$. 
We shall always denote the opposite $\overline{p_{(a,o)}}$ of the path 
$p_{(a,o)}$ in $P_o$ by $p_{(o,a)}$.\\ 
\indent A deformation of a path $p$ is a path obtained either by 
replacing two subsequent 1-simplices  $\partial_0c*\partial_2c$ of $p$
by $\partial_1c$, or by replacing a 1-simplex $\partial_1c$ of $p$ by 
$\partial_0c*\partial_2c$, where $c \in \Si_2(K)$.   
Two paths $p$ and $q$ are \emph{homotopy equivalent}, written  $\sim$,  if 
one can be obtained from the other by a finite sequence of  deformations. 
We shall denote the homotopy class of a path $p$ by $[p]$. 
We then define the \emph{first homotopy group} of $K$, with base point $o\in\Si_0(K)$, as 
\[
 \pi_1^o(K):=\{p:o\to o\} /\sim \ \  , \ \ \ [p]\cdot [q]:= [p*q] \ . 
\] 
Note that, if $\rf:K\to P$ is a poset morphism, then 
the induced  mappings between the set of paths (see above) preserves homotopy equivalence. 
So, by setting  
\[
\rf_*([p]):= [\rf(p)] \ , \qquad [p]\in\pi^o_1(K) \ , 
\]
one has that $\rf_*:\pi^o_1(K)\to \pi^{\rf(o)}_1(P)$ is a group morphism.
Now, since we consider only pathwise connected posets, 
the first homotopy group does not 
depend, up to isomorphism, on the choice of the base point; 
this isomorphism class, written $\pi_1(K)$, is the \emph{fundamental group} of $K$. 
We shall say that $K$ is simply connected whenever $\pi_1(K)$ is trivial.\\
\indent It turns out that a poset $K$ is simply connected if it  is upward directed, because 
there is a contracting homotopy \cite{Rob}, but also if it is  downward directed 
because the fundamental group of a poset  and that 
of its dual are equivalent \cite{RRV}.  A useful result is the 
following (see \cite{Ruz}): when $K$ is a base of neighbourhoods of a space $X$ consisting of arcwise 
and simply connected subsets, $\pi_1(K)$ is isomorphic to the 
homotopy group $\pi_1(X)$. This has important consequences for the present
paper.\\ 
\indent In the following we shall also use  the {\em nerve} $N_*(K)$ of $K$.  
Considering the poset $K$ as a category (taking the elements of $K$ as objects and 
the inclusions as arrows), a 0-simplex of the nerve is an object of this category;
a 1-simplex is an arrow; an $n$-simplex  is a composition of $n$ arrows. 
The nerve  can be realized as a subsimplicial set   of $\Si_*(K)$.
Clearly  $0$-simplices of the nerve are nothing but that $0$-simplices of $\Si_0(K)$.  
For $n\geq 1$, the elements of $N_n(K)$ are those elements $x$ of $\Si_n(K)$ 
whose vertices $x^i:= \partial_{012\cdots(i-1)  (i+1)\cdots (n-1)n}x$, for $i=0,1,\ldots, n$, satisfy 
the relation
\[
x^0\leq  x^1\leq x^2\leq\cdots \leq x^n=|x| \ . 
\]
In the following  we shall use the following notation: 
by $(ao)\in N_1(K)$  we shall mean the  1-simplex of the nerve with $\partial_1(ao)=o$ and $\partial_0(ao)=a$; clearly  $a\geq o$ and $|(ao)|=a$.\\
\indent  Finally, we  introduce a notion of continuity for symmetries of a poset, 
which is mainly used for  posets arising as a base for the topology of a space
as those introduced in \S \ref{A:a}.  
Some preliminary observations are in order.  Let $G$ be a symmetry group of the poset $K$. Extend the action of the group from the poset $K$ to the simplicial set $\Si_*(K)$ and hence to paths, as done above for morphisms of posets. It is clear, from the definition of symmetry of a poset, that $p\sim q$ if, and only if, $gp\sim gq$. So $g_*:\pi^o_1(K)\to \pi^{go}_1(K)$ is a group isomorphism. 
Now, a \emph{continuous symmetry group} is a topological symmetry group 
$G$ of  $K$ such that, for any path $p$ and
$a_0,a_1\in K$ with $\partial_0p < a_0$, $\partial_1p < a_1$, 
there is a neighbourhood $U_e$ of the identity $e \in G$ 
such that $g\partial_0p\leq  a_0$, $g\partial_1p\leq a_1$, and 
\begin{equation}
\label{Ab:1}
 (a,g\partial_0p)*gp*\overline{(o,g\partial_1p)} \sim 
 (a,\partial_0p)*p*\overline{(o,\partial_1p)} \, 
\end{equation}
for any $g\in U_e$. In this equation $(o,g\partial_1p)$ denotes the 1-simplex of the nerve 
having $g\partial_1p$ has 1-face and $o$ as 0-face, that is, $\partial_1(o,g\partial_1p)=g\partial_1p$
and $\partial_0(o,g\partial_1p) =o$, and $\overline{(o,g\partial_1p)}$ denotes the 
opposite  of the 1-simplex  $(o,g\partial_1p)$. 
%
The meaning of (\ref{Ab:1}) is that, in the limit $g \to e$, the path $gp$ becomes homotopic 
to $p$, up to rescaling the starting and the ending points.
%
%
Examples of posets 
having a continuous symmetry group are those described in \S \ref{A:a}, arising as a base for the topology 
of a $G$-space.
%

%
%
%
%
%

\section{Abstract nets of $\rC^*$-algebras}
\label{B}
We develop the abstract theory of nets of $\rC^*$-algebras over posets, focusing 
on those aspects of the theory involving the question 
of representability of nets. We define some basic notions concerning nets of $\rC^*$-algebras  
and give several examples  motivating these definitions.  
Apart from the examples coming from the algebraic quantum field theory, we give new 
examples of nets of $\rC^*$-algebras over any poset.\\
\indent The first important result concerns a particular class of nets, those that we call $\rC^*$-net bundles, which are classified in terms of $\rC^*$-dynamical systems having as group the first homotopy group of the poset. 
The importance of this result relies on two related facts. First, as we shall see in 
\S \ref{C}, this result implies that  any $\rC^*$-net bundle can be faithfully represented.  
Secondly, we shall prove that any net of $\rC^*$-algebras over a poset embeds 
in a unique $\rC^*$-net bundle: the enveloping  net bundle.  
So, the existence of faithful representations turns out to be equivalent  
to the  faithfulness of the embedding of the net into the enveloping net bundle. 
Nets satisfying the latter property are called injective nets.


\subsection{Basic definitions}
\label{Ba}

A \emph{net} of $\rC^*$-algebras $(\cA,\jmath)_K$ is defined by 
a poset $K$, a correspondence $\cA : o \to \cA_o$ associating  
a unital $\rC^*$-algebra $\cA_o$ to any $o \in K$, \emph{the fibre over $o$}, and 
a family $\jmath_{oa}:\cA_a\to\cA_o$, with $a\leq o$, 
of unital faithful $^*$-morphisms, the \emph{inclusion maps}, satisfying the 
\emph{net relations} 
\[
{\jmath}_{oa} \circ  {\jmath}_{ae} = {\jmath}_{oe} \ ,  \qquad e \leq a \leq o \ . 
\]
Whenever the inclusion maps are all $^*$-isomorphisms
we say that $(\cA,\jmath)_K$ is a \emph{$\rC^*$-net bundle}.
If $P \subset K$, then restricting $\cA$ and $\jmath$ to elements of $P$
yields a net called the \emph{restriction of $(\cA,\jmath)_K$ to $P$}, that we denote by $(\cA,\jmath)_P$.
\begin{remark}
 
Some observations are in order.\smallskip  

\noindent 1. According to the above definition, a net of $\rC^*$-algebras is
a pre-cosheaf of $\rC^*$-algebras.  We are adopting the practice  in 
algebraic quantum field theory of calling these objects nets of $\rC^*$-algebras even if it properly 
only applies when  the poset is  upward directed. \smallskip
 
\noindent 2. The term \emph{net bundle} derives from the fact that, as indicated 
in \cite{RRV}, these objects  are fibre bundles over posets  where  geometrical concepts like 
connections and their curvatures can be introduced and analyzed. 
Moreover, when $K$ is a good base for the topology of a space $X$, the category of
net bundles is, in essence, a (non-full) subcategory of the category of bundles on $X$ 
in the usual sense \cite{RRV}.\smallskip

\noindent 3. Note that, when we have a $\rC^*$-net bundle $(\cA,\jmath)_K$,   
each ${\jmath}_{oa}$, $a \leq o$, is invertible, 
and it makes sense to consider the inverses ${{\jmath}_{oa}}^{-1}$.  
To be concise we will write ${\jmath}_{ao} := {{\jmath}_{oa}}^{-1}$.
\end{remark}

A \emph{morphism} $(\phi,\rf):(\cA,\jmath)_K\to (\cB,\imath)_P$ 
of nets of $\rC^*$-algebras is a pair $(\phi,\rf)$, 
where $\rf:K\to P$ is a morphism of posets
%
%
and $\phi$ is a family 
$\phi_o : \cA_o \to \cB_{\rf(o)}$, $o\in K$,  
of $^*$-morphisms fulfilling the relation 
\[
\phi_o\circ \jmath_{oa}= \imath_{\rf(o)\rf(a)}\circ \phi_a \ , \qquad a\leq o \ . 
\]
We say that $(\phi,\rf)$ is \emph{unital} morphism whenever all $\phi_o$ are unital, and 
\emph{faithful on the fibres} whenever $\phi_o$ is faithful for any $o$. Moreover, 
we say that $(\phi,\rf)$ is a  \emph{monomorphism} if $\rf$ is injective and 
$\phi_o$ is faithful for any $o$; it is an \emph{isomorphism} if  
$\rf$ and $\phi_o$, for any $o$,  are isomorphisms.  
There is an obvious composition rule between morphisms: given  $(\psi,\rh):(\cB,\imath)_P\to(\cC,y)_S$ we define  
\[
(\psi,\rh)\circ  (\phi,\rf):= (\psi\circ \phi,\rh\circ \rf) \ : \ 
(\cA,\jmath)_K\to (\cC,y)_S \ ,
\]
leading, in an obvious way, to the category of nets of $\rC^*$-algebras. 
We shall mainly deal with nets over a fixed poset $K$ where  we shall denote 
morphisms of the form $(\phi,\mathrm{id}_K)$ by $\phi$
($\mathrm{id}_K$ is the identity morphism of $K$).\smallskip

The \emph{constant net bundle} with fibre the $\rC^*$-algebra $\rA$ is defined as
the constant assignment 
$\cA^{t}_o := \rA$, $o \in K$,
with inclusion maps
$\jmath^{t}_{\tilde oo} := \mathrm{id}_\rA$ for all $o \leq \tilde o$;
we say that the net is \emph{trivial} if it is isomorphic to a constant net bundle. 
%
Applying the reasoning of \cite{RR,RRV} it can be proved that when $K$ is simply connected
every $\rC^*$-net bundle $(\cA,\jmath)_K$ is trivial; in \S \ref{Bc} we shall give a
proof using dynamical systems. Finally, we say that a net \emph{vanishes} 
if it is isomorphic to the \emph{null net bundle}, that is,  
the $\rC^*$-net bundle having fibre $\{0\}$.
\begin{remark}
It is easily seen from the above definitions that a net  
is a functor from a poset, considered as a category, 
to the category of unital $\rC^*$-algebras 
having  unital faithful ${^*}$-morphisms as  arrows. 
Other types of nets can be obtained by changing the target category,
as follows. \\ 

\indent A \emph{net of Hilbert spaces} $(\cH,U)_K$ is given by the correspondence
assigning Hilbert spaces $\cH_o$, $o \in K$,
and a family of isometries $U_{ao} : \cH_o \to \cH_{a}$, $o \leq a$,
fulfilling 
$U_{eo} = U_{ea} \, U_{ao}$, $o \leq a \leq e$.
When each $U_{ao}$ is unitary we say that $(\cH,U)_K$ is a \emph{Hilbert 
net bundle}.\\ 
\indent A \emph{net of locally compact groups} $(\cG,\jmath)_K$ is defined
assigning locally compact groups $\cG_o$, $o \in K$, and continuous 
group monomorphisms $\jmath_{ao}:\cG_o\to\cG_a$, $o\leq a$,  fulfilling
$\jmath_{eo} = \jmath_{ea} \circ \jmath_{ao}$, $o \leq a \leq e$.
We say that $(\cG,\jmath)_K$ is a group net bundle whenever each $\jmath_{ao}$
is an isomorphism. 
\end{remark}

Let $G$ be symmetry group of $K$ in the sense of \S \ref{A:a}. 
A net of $\rC^*$-algebras $(\cA,\jmath)_K$ is \emph{$G$-covariant} if for 
any $g\in G$ there is  a family $\alpha_o^g:\cA_o\to\cA_{go}$, $o\in K$,  
of $^*$-isomorphisms such that 
\begin{equation}
\alpha_{go}^h \circ \alpha_o^g = \alpha_o^{hg} \ , \qquad g,h\in G \ , 
\end{equation}
and 
\begin{equation}
\alpha_o^g \circ {\jmath}_{oa} = {\jmath}_{go\,ga}\circ \alpha_a^g  
\ , \qquad 
a\leq o \ .  
\end{equation}
Let $(\cB,\imath,\beta)$ be a $G$-covariant net. A morphism 
$\phi : (\cA,\jmath)_K \to (\cB,\imath)_K$
is said to be $G$-covariant whenever 
$\beta^g_o \circ \phi_o = \phi_{go} \circ \alpha^g_o$,
$\forall o \in K$, $g \in G$. When $G$ is a continuous symmetry group of $K$  (see (\ref{Ab:1})) 
we assume that the action $\alpha$ on a covariant net $(\cA,\jmath,\alpha)_K$ 
is \emph{continuous}. This amounts to saying that if  $\{g_\lambda\}_{\Lambda}$ is a net in $G$ converging to the identity of the group,
then for any  $o\in K$,  and for any $a\in K$ with $a > o$,  there exists  $\lambda_a\in \Lambda$ 
such that $g_{\lambda}o\leq a$ for any $\lambda\in\Lambda$ with 
$\lambda\geq \lambda_a$ and  
\begin{equation}
\norm{\jmath_{a\, g_\lambda o}\circ \alpha^{g_\lambda}_o(A)-\jmath_{ao}(A)} \to 0 \ ,\qquad \forall A\in\cA_o \ . 
\end{equation}
We stress that this notion of continuity is intended for  posets arising as the base for the topology 
of a topological space acted upon, continuously, by a group (see \S \ref{A:a}). 
In this cases, one can easily see that when the poset is upward directed this notion of continuity reduces
to the usual notion in   algebraic quantum field theory (see \cite{Haa}).\\
\indent Finally, when  $K$ is endowed with a causal disjointness 
relation $\perp$, we say that a net of $\rC^*$-algebras $(\cA,\jmath)_K$ is  \emph{causal}
if, for any $o_1,o_2\subseteq o$ with $o_1\perp o_2$, there holds 
\begin{equation}
 [{\jmath}_{oo_1}(A_1), {\jmath}_{oo_2}(A_2)]=0 \ , 
\end{equation}
for any $A_1\in\cA_{o_1}$ and  $A_2\in\cA_{o_2}$, where $[\cdot,\cdot]$ denotes the commutator.\\
%


\subsection{Examples} 
\label{Bb}
Many of the  examples 
of nets of $\rC^*$-algebras  arise  as models of quantum fields over a spacetime, mainly 
\emph{via} the Weyl quantization of the phase space of a classical field.  
For any spacetime mentioned in \S \ref{A:a} (the Minkowski space, the circle $S^1$, 
or an arbitrary globally hyperbolic spacetime) there are causal nets, 
over a suited poset (see \S \ref{A:a}), which are covariant with respect to the symmetries of the spacetime.
For these examples we refer the reader to \cite{Haa,Ara,BrR,Dim,Ve,BGP,BMT,FRS}. 
We also quote  \cite{BFV} for generally locally covariant nets over 
globally hyperbolic spacetimes,  and \cite{BrR,AM} for examples 
of nets over the lattice $\mathbb{Z}^n$. \\
\indent In all the mentioned examples  the representability problem  does not arise: 
the nets are represented on the  \emph{physical} Hilbert space of the model\footnote{For instance, 
models obtained by Weyl quantization are, in general, represented on a Fock space.
More in general, in these models the net  faithfully embeds into the $\rC^*$-algebra 
associated with whole phase space. Faithful representations of this algebra induce 
faithful representations of the net.}.  
However, in order to study model independent aspects of a theory one is led 
to consider abstract nets of $\rC^*$-algebras. Apart from the cases, like  
the Minkowski space, where the poset indexing the net is upward directed (see \S \ref{A:a}),
there is no result about the existence of representations for these nets; 
in these cases the existence of a representation is 
a working assumption.  This motivates our investigation.\\ 
\indent Important examples of nets  outside the context of 
quantum field theory comes from geometric group theory: as said in the introduction, any 
simple complex of groups is  a net of groups (\cite{BHa,GHV,Hae,Sta}). Examples are also 
the systems of $\rC^*$-algebras and  the systems of groups  
underlying the notions of (generalized) amalgamated 
free product \cite{Neu1,Neu2,Bla,Bil}.\\ 
\indent In the next subsections we provide new  examples of nets. 
In particular, we give a combinatorial construction of nets  of $\rC^*$-algebras 
over arbitrary posets.  These nets turn out to be  covariant when the posets have a symmetry group.  Aspects related to the causal structure will be discussed in a future work \cite{CRV}.


\subsubsection{Nets of groups of loops}
\label{Bba}

We introduce a way of deforming a path of a poset 
which turns out to be weaker than that underlying the homotopy equivalence relation
This new deformation allows us to construct examples of nets of discrete groups 
(in the present section)  and of $\rC^*$-algebras (later) over any poset. \bigskip

Given two paths $p$,$q$ with the same endpoints. 
We say that $q$ is a $w$-\emph{deformation} of $p$, and conversely, if
\begin{itemize}
 \item[(1)] either $q$ is obtained by inserting into $p$ a degenerate 1-simplex;           
 \item[(2)] or by replacing two consecutive 1-simplices
 $\partial_0c*\partial_2c$ of the path $p$ by $\partial_1c$ where 
   $c\in N_2(K)$ (a 2-simplex of the nerve);
\item[(3)] or $q$ can be obtained by replacing two consecutive 1-simplices
 $\overline{b}* b$  of the path  $p$ by $\si_0\partial_1b$. 
\end{itemize}
Two paths with the same endpoints are \emph{$w$-equivalent} if one can be obtained from 
the other by a finite sequence of $w$-deformations.\bigskip

This is an equivalence relation weaker than homotopy since  
any pair of  w-equivalent paths are also homotopy equivalent. 
The following example points out the difference between these two 
relations. Consider a 1-simplex $b\in\Si_1(K)$, and the associated path
\[
p_b:= (\overline{|b|\partial_0b})*(|b|\partial_1b)  \ , 
\]
where $(|b|\partial_1b)$ is the 1-simplex of the nerve of $K$ (see \S \ref{A:b}) 
having $1$-face  $\partial_1b$
and $0$-face $|b|$; while $(\overline{|b|\partial_0b})$ is the opposite 
in $\Si_1(K)$ of the 1-simplex of the nerve  $(|b|\partial_0b)$. Then 
$b$ is homotopic to $p_b$: the relations 
\[
\partial_0c_b:= (\overline{|b|\partial_0b}) \ , \ \ 
\partial_2c_b:=(|b|\partial_1b) \ , \ \ \partial_1c_b:= b \ , \ \ 
|c_b|:=|b| \ , 
\]
define a 2-simplex $c_b$ of $\Si_2(K)$. However 
\[
b\sim_w p_b \ \ \iff \ \ b\in N_1(K) \ .  
\]
In fact if $b$ does not belong to the nerve of $K$, then $c_b$ does not belong 
to the nerve of $K$, and this  implies that $b\not\sim_w p_b$. Conversely if  
$b$ belongs to the nerve, then  $\partial_0b=|b|$.  So $(\overline{|b|\partial_0b})=\si_0\partial_0b$ (the degenerate 1-simplex associated with the 0-face 
of $b$) and $b= (|b|\partial_1b)$. Therefore  $c_b\in N_2(K)$ and 
$b\sim_w p$ by property (2) of the above definition.\\ 
\indent Now it is easily seen that  $w$-equivalence is compatible with the composition of paths 
and with the operation of taking the opposite of a path. This allows us 
to associate two groups with any $o$ of $K$. 
The first group is defined as the quotient set  of the set of loops over $o$ 
with respect to the $w$-equivalence: 
\begin{equation}
  \Lambda_o:=\{p:o\to o \ \}/\sim_w  \ .   
\end{equation}
The product is defined as 
$[p]_w\cdot [q]_w:=[p*q]_w$. We call this group 
the \emph{$w$-group of loops} over $o$. The second group is the subset of  $\Lambda_o$
defined by
\begin{equation}
  \Lambda^l_o:=\{[p]_w\in \Lambda_1(o) \ | \  \exists q\in [p]_w \ \mbox{ such that } 
  |q|\subseteq o\} \ , 
\end{equation}
with the same product as $\Lambda_o$. In words $\Lambda^l_o$ is the subset 
of those elements $[p]_w$ of  $\Lambda_o$ for which there is at least 
one path in the equivalence class of $[p]_w$ whose support is contained in $o$.
This is a subgroup  because 
composition of paths whose support is contained in $o$ leads to a path supported in $o$. 
We call  $\Lambda^l_o$ the \emph{w-group of loops supported in $o$}.\\ 
\indent The next step is to prove that these groups form nets of discrete groups. 
Given an inclusion $o\leq a$, define 
\begin{equation}
\lambda_{ao}([p]_w):= [(ao)*p*(\overline{ao})]_w \ , \qquad [p]_w\in \Lambda_1(o) \ .  
\end{equation}
Here $(ao)$, as above,  is the 1-simplex of the nerve associated with the inclusion $o\leq a$;  
instead $(\overline{ao})$ is the opposite, in $\Si_1(K)$, of $(ao)$ (see the preliminaries). \\ 
\indent It is clear that $\lambda_{ao}:\Lambda_o\to\Lambda_a$. Moreover, it is easily seen from the definition  that  
$\lambda_{ao}:\Lambda^l_o\to\Lambda^l_a$. Now, for any $[q]_w,[p]_w\in\Lambda_o$ we have 
\begin{align*}
\lambda_{ao}([p]_w)\,\lambda_{ao}([q]_w )& = [(ao)*p*(\overline{ao})]_w\cdot [(ao)*q*(\overline{ao})]_w \\
& = [(ao)*p*(\overline{ao})*(ao)*q*(\overline{ao})]_w \\
&  =  [(ao)*p*q*(\overline{ao})]_w =  \lambda_{ao}([p*q]_w)\ , 
\end{align*}
because of property $(3)$ of the definition of $w$-deformation. Moreover for any 
inclusion $e\leq o\leq a$
\[
\lambda_{ao}(\lambda_{oe}([p]_w)) = [(ao)*(oe)*p*(\overline{oe})*(\overline{ao})]_w
= [(ae)*p*(\overline{ae})]_w = \lambda_{ae}([p]_w) \ , 
\]
for any $[p]_w\in\Lambda_e$,  
because of property $(2)$ of the definition of $w$-deformation. Finally    
if $[p]_w\in\Lambda_a$, then $[(\overline{ao})*p*(ao)]_w\in\Lambda_o$ for $o\leq a$, and  
$\lambda_{ao}([(\overline{ao})*p*(ao)]_w) = [p]_w$. This proves that 
$\lambda_{ao}:\Lambda_o\to\Lambda_a$ is group isomorphism making $(\Lambda,\lambda)_K$ 
a net bundle of discrete groups and $(\Lambda^l,\lambda)_K$ a net of discrete groups.
\begin{remark}
\label{Bba:1}
In general the triple $(\Lambda^l,\lambda)_K$ is  \emph{is not a net bundle}. For,
it is enough to consider inclusions  $o,\tilde o < a$ such that $o$ and $\tilde o$ are not
related. In this case one can easily construct elements of $\Lambda^l_a$ not belonging to 
the image of $\Lambda^l_o$ by $\lambda_{ao}$.
\end{remark}
Assume that $K$ has a symmetry group $G$. Extend the action of the symmetry group from the poset
$K$ to the simplicial set $\Si_*(K)$, and hence to paths, as done 
in \S \ref{A:a} for morphisms of posets. Note that 
$p\sim_{w}q$ if, and only if, $gp\sim_w gq$ for any $g\in G$. This is quite obvious 
from  how $g$ is defined on paths. Therefore, 
\begin{equation}
\label{Bba:2}
g_*([p]_w):=[gp]_w \ , \qquad [p]_w\in\Lambda_o \ ,  
\end{equation}
is well defined, and  
$g_*:\Lambda_o \to \Lambda_{go}$ and
$g_*:\Lambda^l_o \to \Lambda^l_{go}$ are group isomorphisms. Moreover, given 
$o\leq a$ we have 
\[
g_*(\lambda_{ao}([p]_w)) = g_*([(ao)*p*(\overline{ao})]_w) =
[g(ao)*gp*\overline{g(oa)}]_w = \lambda_{g(ao)}(g_*([p]_w)) \ ,
\]
for any $[p]_w\in\Lambda_o$.  Summing up, we have the following result. 
\begin{proposition}
\label{Bba:3}
Let $K$ be a poset. 
\begin{itemize}
\item[(i)] Then $(\Lambda,\lambda)_K$ and $(\Lambda^l,\lambda)_K$ are, respectively, 
a net bundle and a net of discrete groups with a monomorphism 
$i:(\Lambda^l,\lambda)_K\to (\Lambda,\lambda)_K$ 
defined by the inclusions $i_o:\Lambda^l_o\to\Lambda_o$. 
\item[(ii)] If  $K$ has  a symmetry group $G$, then 
$i:(\Lambda^l,\lambda)_K\to (\Lambda,\lambda)_K$ is a monomorphism of $G$-covariant nets. 
\end{itemize}
\end{proposition}
\begin{proof}
It is clear that the mapping $i:\Lambda^l_o\to\Lambda_o$ defined 
by $i_o([p]_w)=[p]_w$, with $[p]_w\in\Lambda^l_o$,  is a monomorphism. 
Moreover 
$g_*\circ i_o([p]_w)=[gp]_w = i_{go}([gp]_w)=i_{go}\circ g_*([p]_w)$ for any 
$[p]_w\in \Lambda^l_o$. This completes the proof.
\end{proof}

\subsubsection{Nets of $\rC^*$-algebras from nets of discrete groups}
\label{Bbc}

In the present section we make use of the functor assigning the group $\rC^*$-algebra
to construct nets of $\rC^*$-algebras starting from nets of discrete groups.\bigskip

Let $G$ be a discrete group and $\rC^*(G)$ denote the (full) group $\rC^*$-algebra; 
this is a unital $\rC^*$-algebra defined as the enveloping $\rC^*$-algebra of 
the convolution algebra $\ell^1(G)$.  The unital $^*$-algebra $\mathbb{C}(G)$ of those 
functions $f:G\to\mathbb{C}$ which are a finite linear combination 
$f=\sum_g f(g) \, \delta_g$ of  Kronecker delta functions 
is a  dense subset of $\ell^1(G)$. 
In particular $\delta_g*\delta_h=\delta_{gh}$. \\
\indent Given a group  morphism $\si:G\to H$, for any $f\in\mathbb{C}(G)$ let 
\[
 \tilde\si(f):= \sum f(g) \delta_{\rho(g)} \ . 
\]
This defines a morphism $\tilde\si: \mathbb{C}(G)\to \mathbb{C}(H)$
in the $\ell^1$-norm. Moreover it is an isometry (isomorphism) when $\si$ is injective (is an isomorphism). 
Therefore $\tilde\si$ lifts to a continuous morphism  from $\ell^1(G)$ into $\ell^1(H)$. 
By  the universality property of the enveloping $\rC^*$-algebra there exists 
a unique morphism $\rC^*(\si):\rC^*(G)\to\rC^*(H)$ such that 
\[ 
 \rC^*(\si)\circ \iota_G = \iota_H\circ \tilde \si \ , 
\]
where $\iota$ denotes the embedding of the $\ell^1$ algebra into the enveloping 
$\rC^*$-algebra. 
Therefore, the mapping
\[
\rC^* : G \mapsto \rC^*(G) \ \ , \ \ \si \mapsto \rC^*(\si)
\]
is a functor. When $\si$ is a group isomorphism $\rC^*(\si)$ is, clearly, a ${^*}$-isomorphism,
nevertheless we are interested in the case where  $\si$ is simply a monomorphism,
and for a  locally compact group  $\rC^*(\si)$ is
not injective in general. On the other hand, assuming that $G,H$ are discrete we find that for every unitary 
representation $\pi$ of $G$ there is an isometry $V\in(\pi, \mathrm{ind}(\pi) \circ \si)$,
where $\mathrm{ind}(\pi)$ is the representation of $H$ induced by $\pi$
\footnote{We are grateful with A. Valette for drawing  our attention to this fact, and to
a counterexample when the group involved are not discrete.}
(see for instance \cite{Fol}). By the definition of the enveloping $\rC^*$-algebra, 
and the relation between the unitary representations of a locally compact group 
and  non degenerate representations of the $\ell^1$ algebra, 
$\rC^*(\si)$ is a faithful ${^*}$-morphism.\\ 
\indent The functor $\rC^*$ allows one to construct in the obvious way a net of $\rC^*$-algebras 
$(\rC^*(\cG),\rC^*(\jmath))_K$, 
$\rC^*(\cG)_o := \rC^*(\cG_o)$, 
$\rC^*(\jmath)_{o'o} := \rC^*(\jmath_{o'o})$,
$o \leq o' \in K$,
starting from the net of discrete groups $(\cG,\jmath)_K$.
Clearly, $(\rC^*(\cG),\rC^*(\jmath))_K$ is a net bundle when $(\cG,\jmath)_K$ is a net bundle,
and this yields the following result. 
\begin{proposition}
\label{Bbc:1}
Given a poset $K$, 
let $(\Lambda^l,\lambda)_K$ and  $(\Lambda,\lambda)_K$ 
be  the net of discrete groups defined in \S \ref{Bba}. Then
$(\rC^*(\Lambda^l),\rC^*(\lambda))_K$ and $(\rC^*(\Lambda),\rC^*(\lambda))_K$
are respectively a net of $\rC^*$-algebras and a $\rC^*$-net bundle, and 
\[
\rC^*(i):(\rC^*(\Lambda^l),\rC^*(\lambda))_K\to (\rC^*(\Lambda),\rC^*(\lambda))_K
\]
is a unital monomorphism. 
%
%
If $K$ has a symmetry group $G$, then all the above nets and morphisms are $G$-covariant. 
\end{proposition}
\begin{proof}
The first assertion follows as  $\rC^*$ is a functor. 
In a similar fashion, to prove covariance we note that by Prop.\ref{Bba:3}
the two  nets of groups are $G$-covariant, and, given $g\in G $, define
\[
c_g := \rC^*(g^*) : \rC^*(\Lambda)_o\to \rC^*(\Lambda)_{go}
\]
where $g^*:\Lambda_o\to \Lambda_{go}$ is defined in (\ref{Bba:2}). 
It is easily verified that this makes our nets $G$-covariant, as desired. 
\end{proof}
\begin{remark}
\label{Bbc:2}
With the applications to algebraic quantum field theory in mind,
this is a promising result: it  allows us 
to construct nontrivial examples of covariant nets of $\rC^*$-algebras over any spacetime, taking 
as poset a suitable base of neighbourhoods for the topology of the spacetime 
encoding the causal structure and the global symmetry of the spacetime, as 
indicated in the preliminaries. However, this result, as it stands,  is not complete.
First, the above nets are not causal. Secondly, if $G$ is a continuous symmetry for the poset (\ref{Ab:1}), 
then the action of $G$ on the net of $\rC^*$-algebras defined in Prop.\ref{Bbc:1} 
is not continuous according to the definition given in \S \ref{Ba}. This is so 
because $w$-equivalence, used to define the groups of loops, is weaker than homotopy  
equivalence. We will fill these gaps in a forthcoming paper  \cite{CRV}.
\end{remark}



\subsection{The holonomy dynamical system} 
\label{Bc}

We now focus on net bundles.
We shall see that the fibres of a net bundle are acted upon 
by the homotopy group of the poset.  This in turns will lead to 
an equivalence between the category of net bundles over a poset and that of  
$\pi^o_1(K)$-dynamical systems for any $o\in K$. We shall work using 
morphisms over different posets, a scenario that will appear in the
forthcoming paper \cite{RVs1} and is useful in generally locally covariant field theory
(\cite{BFV}).
\bigskip

We start by making some preliminary definitions. 
Let $(\rA,G,\alpha)$ and $(\rB,H,\beta)$ be $\rC^*$-dynamical systems; a morphism 
from $(\rA,G,\alpha)$ to $(\rB,H,\beta)$ is pair $(\eta,\rp)$ where 
$\eta:\rA\to\rB$ is a ${^*}$-morphism and $\rp:G\to H$ is a group morphism satisfying the relation 
\begin{equation}
\label{Bc:1}
\eta \circ \alpha_g= \beta_{\rp(g)}\circ \phi \ , \qquad g\in G \ . 
\end{equation}
We say that $(\eta,\rf)$ is an \emph{isomorphism} when $\eta$ is a ${^*}$-isomorphism and 
$\rf$ is a group isomorphism. When $G=H$ we shall denote morphisms of the form 
$(\eta,\mathrm{id}_G)$ by $\eta$.\smallskip

Now, consider a net bundle $(\cA,{\jmath})_K$ and define 
\begin{equation}
\label{Bc:2}
 {\jmath}_b:= {\jmath}_{\partial_0b|b|}\circ {\jmath}_{|b|\partial_1b} \ , \qquad b\in\Si_1(K) \ , 
\end{equation}
where ${\jmath}_{\partial_0b|b|}:= {\jmath}^{-1}_{|b|\partial_0b}$. Since the inclusion 
maps are isomorphisms, we have a field 
\[
\Si_1(K)\ni b\to {\jmath}_b \in \mathrm{Iso}(\cA_{\partial_1b},\cA_{\partial_0b}) \ ,
\] 
satisfying  the \emph{1-cocycle} 
equation
\begin{equation}
\label{Bc:3}
 {\jmath}_{\partial_0c}\circ {\jmath}_{\partial_2c}={\jmath}_{\partial_1c} \ , \qquad c\in\Si_2(K) \ , 
\end{equation}
where $\mathrm{Iso}(\cA_{\partial_1b},\cA_{\partial_0b})$ is the set of ${^*}$-isomorphisms 
from $\cA_{\partial_1b}$ to $\cA_{\partial_0b}$. 
Extend this 1-cocycle from 1-simplices to paths by setting  
${\jmath}_p:={\jmath}_{b_n}\circ \cdots \circ{\jmath}_{b_2}\circ{\jmath}_{b_1}$ 
for any path $p:=b_n*\cdots *b_2*b_1$. Then the 1-cocycle equation implies \emph{homotopy invariance}, 
that is, ${\jmath}_p={\jmath}_q$ whenever $p\sim q$ (see \cite{Ruz}). 
In this way, fixing an element $o$ of $K$ and defining $\cA_* := \cA_o$ yields the action 
\begin{equation}
\label{Bc:4}
\jmath_* : \pi_1^o(K) \to {\mathrm{Aut}} \cA_*
\ \ , \ \
\jmath_{*, [p]}:= {\jmath}_p 
\ \ , \ \
[p]\in\pi_1^o(K)
\ .
\end{equation}
We call
$( \cA_* , \pi_1^o(K) , \jmath_* )$  
the \emph{holonomy dynamical system} associated with the net bundle $(\cA,\jmath)_K$. 
A different choice of the base element $o$ leads to an isomorphic  
dynamical system. 
If $(\phi,\rf) : (\cA,\jmath)_K \to (\cB,\imath)_S$ is a morphism, 
there is an induced morphism $\rf_* : \pi_1^o(K) \to \pi_1^{\rf(o)}(S)$, and defining
\[
( \phi_* , \rf_* ) : 
( \cA_* , \pi_1^o(K) , \jmath_* ) \to 
( \cB_* , \pi_1^{\rf(o)}(S) , \imath_* )
\ \ , \ \
\phi_* := \phi_o
\ ,
\]
yields a morphism of dynamical systems, in fact
\[
\phi_* \circ \jmath_{*,[p]}  = 
\phi_o \circ \jmath_p = 
\imath_{\rf_*(p)} \circ \phi_o = 
\imath_{*,\rf_*[p]}\circ \phi_* \ . 
\]
%

\begin{remark}
\label{Bc:5}
Just a comment on the terminology: as said before, a net bundle can be seen 
as a fibre bundle over the underlying poset; the fibre over $o$ is nothing but  the 
algebra ${\cA_o}$, and the inclusion maps ${\jmath}$ define a connection of the fibre bundle. 
So, for any loop $p$ over $o$, the automorphism ${\jmath}_p$ is the parallel transport 
along  $p$ for the connection ${\jmath}$, i.e., the holonomy.
\end{remark}

Now, we want to prove that the mapping 
\begin{equation}
\label{Bc:6}
(\cA,\jmath)_K \mapsto ( \cA_* , \pi_1^o(K) , \jmath_* )
\end{equation} 
is bijective up to isomorphism. Thus we have to find an inverse
(up to isomorphism).  To this end, we fix a path frame 
$P_o := \{ p_{(a,o)} , a \in K \}$ over $o\in K$ (see \S \ref{A:b}) 
and, given the dynamical system $(\rA,\pi_1^o(K),\alpha)$, define
\begin{equation}
\label{Bc:7}
 \alpha_{*,\tilde a a }:= \alpha_{[p_{(o,\tilde a)}*(a\tilde a)*p_{(a,o)}]} 
 \in {\mathrm{Aut}}\rA
 \ , \qquad 
 a\leq \tilde a \ , 
\end{equation}
where, by convention, $p_{(o,\tilde a)}$ denotes the opposite $\overline{p_{(\tilde a,o)}}$ of 
$p_{(\tilde a,o)}$. Observe that 
\begin{align*}
 \alpha_{*,\tilde a a } \circ \alpha_{*,ae} & = 
 \alpha_{[p_{(o,\tilde a)}*(\tilde aa)*p_{(a,o)}]}\circ 
 \alpha_{[p_{(o,a)}*(ae)*p_{(e,o)}]} 
  = \alpha_{[p_{(o,\tilde a)}*(\tilde aa)*p_{(a,o)}* p_{(o,a)}*(ae)*p_{(e,o)}]} \\
 & = \alpha_{[p_{(o,\tilde a)}*(\tilde aa)*(ae)*p_{(e,o)}]} 
  = \alpha_{[p_{(o,\tilde a)}*(\tilde ae)*p_{(e,o)}]} \\
  & = 
 \alpha_{*,\tilde a e} \ ,
\end{align*}
where we have used the homotopy equivalence of the paths involved. 
So $(\rA_*,\alpha_*)_K$, where $\rA_*$ is the constant 
assignment $\rA_{*,a} := \rA$ for any $a \in K$, is a $\rC^*$-net bundle. 
We call $(\rA_* , \alpha_* )_K$ the \emph{net bundle associated} 
with the dynamical system $(\rA,\pi_1^o(K),\alpha)$. It is easily seen that 
a different choice of  path frame leads to isomorphic  
net bundles. \\ 
\indent Passing to the level of morphisms requires more attention. 
Consider a poset morphism
$\rf : K \to S$. 
Take a path frame $P_o$ in $K$, and choose, in $S$ a path frame 
$P'_{\rf(o)}:=\{ q_{(a',\rf(o))},\  a'\in S\}$  whose elements satisfy the 
condition
\[
 q_{(\rf(a),\rf(o))} = \rf(p_{(a,o)})  \ , \qquad a\in K \ . 
\]
So, this path frame is an extension of the image  $\rf(P_o)$  of $P_o$ to $S$.  
Such an extension exists since no other 
restriction is imposed on  a path frame 
but that $q_{(\rf(o),\rf(o))}$ be homotopic to the degenerate 1-simplex $\si_0\rf(o)$. 
Since $q_{(\rf(o),\rf(o))}=\rf(p_{(o,o)})$, this condition is automatically fulfilled because 
poset morphisms preserve homotopy equivalence (see \S \ref{A:b}). \\
\indent Given a morphism
\smash{$(\eta,\rf_*) : (\rA,\pi_1^o(K),\alpha) \to (\rB,\pi_1^{\rf(o)}(S),\beta)$} 
of dynamical systems, we construct the net bundle $(\rB_*,\beta_*)_S$ using $P'_{\rf(o)}$ 
as above\footnote{We  consider only morphisms of the form $(\eta,\rf_*)$ where 
$\rf_*$ is the extension of a poset morphism $\rf:K\to S$ to the corresponding homotopy groups. 
This suffices for our purpose.}.  Define 
$\eta_{*,a} := \eta$, $\forall a \in K$,
giving
\[               
\eta_{*,\tilde a} \circ \alpha_{*,\tilde a a }  = 
\eta \circ \alpha_{[p_{(o,\tilde a)}*(\tilde aa)*p_{(a,o)}]} =
\beta_{\rf_*[p_{(o,\tilde a)}*(\tilde aa)*p_{(a,o)}]} \circ \eta =
\beta_{*,\rf(\tilde a) \rf(a)} \circ \eta_{*,a}
\ .
\]                
Thus $\eta_*$ is a morphism of net bundles. We are ready to prove 
that the mapping 
\begin{equation}
\label{Bc:8}
(\rA,\pi_1^o(K),\alpha) \mapsto (\rA_* , \alpha_* )_K
\end{equation}
%
%
is, up to isomorphism, the inverse of (\ref{Bc:6}). This amounts to showing that 
the dynamical system $(\rA,\pi_1^o(K),\alpha)$  and the net bundle $(\cA,\jmath)_K$ are, respectively, isomorphic to the dynamical system 
$(\rA_{**},\pi_1^o(K),\alpha_{**})$ and to  the net bundle 
$(\cA_{**},\jmath_{**})_K$ defined with respect to a fixed pole $o$ and to a fixed 
path frame $P_o$ of $K$. 
In the first case, by construction we have
$\rA_{**} = \rA_{*,a} \equiv \rA$ for all $a \in K$, and, for any loop $q$ over $o$ 
of the form $q=b_n*\cdots *b_1$,
\[
  \alpha_{**,[q]}  = \alpha_{*,q} 
= \alpha_{
     [p_{(o,\partial_0b_n)}*b_n *p_{(\partial_1b_n,o)}* 
     \cdots * 
     p_{(o,\partial_0b_1)}*b_1 *p_{(\partial_1b_1,o)}]} 
 = \alpha_{[p_{(o,o)}*q*p_{(o,o)}]}
 = \alpha_{[q]} \ ,
\] 
thus $(\rA_{**},\pi_1^o(K),\alpha_{**}) = (\rA,\pi_1^o(K),\alpha)$.
In the second case, we define the family of $^*$-isomorphisms
\begin{equation}
\label{Bc:9}
\tau_a:= {\jmath}_{p_{(o,a)}} : \cA_a \to \cA_o \ , \qquad a\in K \ ; 
\end{equation}
to prove that $\tau$ defines an isomorphism 
$\tau : (\cA,\jmath)_K \to (\cA_{**},\jmath_{**} )_K$,
we compute
\[
 \tau_a \circ {\jmath}_{ae}  = {\jmath}_{p_{(o,a)}*(ae)} = 
 {\jmath}_{p_{(o,a)}*(ae)*p_{(e,o)}}\circ {\jmath}_{p_{(o,e)}} \\
 = \jmath_{*,[p_{(o,a)}*(ae)*p_{(e,o)}]}\circ {\jmath}_{p_{(o,e)}}
  = {\jmath}_{**,ae} \circ \tau_e  \ ,
\] 
and this shows that $\tau$ preserves the net structures, as desired.
Thus we have proved:
\begin{proposition}
\label{Bc:10}
There exists a correspondence, bijective up to isomorphism, 
between net bundles over a poset and dynamical systems having as group the first homotopy group
of the poset.  
In particular, the category of net bundles over $K$ with morphisms 
$(\phi,\mathrm{id}_K)$ is equivalent to the category 
of  $\rC^*$-dynamical systems with group $\pi^o_1(K)$, for some $o\in K$, with morphisms  
$(\eta, \mathrm{id}_{\pi^o_1(K)})$. 
\end{proposition}
\begin{proof}
The first part of the statement has already been proved. The second part, the categorical
equivalence, follows directly from the above calculations. 
\end{proof}
The preceding analysis has a wider scope than that indicated
in Prop. \ref{Bc:10}. In algebraic quantum field theory one deals with nets 
of $\rC^*$-algebras over a poset associated with a fixed spacetime. 
Thus considering morphisms of nets between different posets  means  
dealing with nets defined over different spacetimes. This is part of  the more general framework of the 
generally locally covariant quantum field theories \cite{BFV}. 
However,  this topic would lead  us too far from the mainline 
of the paper. 
\begin{corollary}
\label{Bc:11}
If $K$ is simply connected then any  net bundle $(\cA,\jmath)_K$ is trivial.
\end{corollary}
\begin{proof}
If $K$ is simply connected then $\pi_1^o(K)$ is trivial and hence only trivial 
$\pi_1^o(K)$-actions occur. On the other hand, by construction, 
the net bundle associated with a dynamical system with trivial action is clearly trivial.
\end{proof}
The above result applies analogously to other categories than that  of
$\rC^*$-algebras, for example, the category  of Hilbert spaces or topological groups.
\begin{example}
\label{Bc:12}
Let $\rf : \pi_1^o(K) \to G$ be a continuous group morphism 
(here $\pi_1^o(K)$ has the discrete topology). If $(\rA,G,\alpha)$
is a dynamical system, then there is an induced net bundle
$(\rA_*,(\alpha \circ \rf)_*)_K$.
In particular:
\begin{itemize}
\item[(i)]  The unitary group $\bU(d)$, $d \in \bN$, acts on the Cuntz algebra 
            $\cO_d$, thus every unitary representation of $\pi_1^o(K)$ induces a 
            net bundle with fibre $\cO_d$.
\item[(ii)] Let $(V,\omega)$ be a symplectic space and 
            $\rf : \pi_1^o(K) \to {\mathrm{Aut}}(V,\omega)$
           a group morphism. Since the symplectic group ${\mathrm{Aut}}(V,\omega)$
           acts by automorphisms on the Weyl algebra $\rW_{(V,\omega)}$, we conclude that
           $\rf$ defines a net bundle with fibre $\rW_{(V,\omega)}$.
\end{itemize}
\end{example}


%
\subsection{The enveloping net bundle and injectivity}
\label{Bd}

As we saw in the previous section net bundles can be efficiently classified in terms of 
dynamical systems and contain interesting geometric informations, thus it is of interest
to characterize those nets of $\rC^*$-algebras admitting an embedding into a $\rC^*$-net bundle.
In a certain sense, this is a slight generalization of the problem of characterizing those nets
admitting an embedding into a \emph{single} $\rC^*$-algebra, as for example the Fredenhagen
universal algebra \cite{Fre}. In the present section we shall show  
that any net of $\rC^*$-algebras embeds into a unique $\rC^*$-net bundle:
the enveloping net bundle.  
We then define to be injective those nets having a faithful embedding  into the  
enveloping net bundle.\bigskip

Let $(\cA,{\jmath})_K$ denote a net of $\rC^*$-algebras. 
Given $o\in K$, we let $\overline{\cA}_o$ be the free unital algebra
generated by the set of symbols
\begin{equation}
\label{Bd:1}
 \{ \, (p,A) \ | \ \partial_0p = o \, , \ A\in\cA_{\partial_1p} \, \}   \ . 
\end{equation}
This is indeed a ${^*}$-algebra: the adjoint is defined on generators by 
\begin{equation}
\label{Bd:2}
(p,A)^*:= (p,A^*) \ , 
\end{equation}  
and extended by anti- multiplicativity and anti- linearity to all of $\overline{\cA}_o$.\smallskip 

 We now add some additional relations. The first set of relations are of algebraic nature : 
\begin{align}
\label{Bd:3}
&  (p,A)\cdot (p,B) = (p,AB) \\
\label{Bd:4}
&  (p,A+B)= (p,A) + (p,B)  \\
\label{Bd:5}
&  (p,\mathbbm{1})=\mathbbm{1} \ , \ \ \ (p,0)=0 \ ,  
\end{align}
which hold for any path $p$ and $A,B\in \cA_{\partial_1p}$. 
The next  two relations encode the net structure and the topology of the poset.  
The first one is  \emph{isotony}: given $\tilde a\geq a$,
\begin{equation}
\label{Bd:6}
(p,{\jmath}_{\tilde aa}(A))= (p*(\tilde a a),A)  \ , 
\end{equation}
for any path $p:\tilde a\to o$ and $A\in{\cA_a}$. The second is  
\emph{homotopy invariance}: if $p\sim q$ then   
\begin{equation}
\label{Bd:7}
 (p,A)= (q,A) \ .  
\end{equation}
With an abuse of notation, we again denote the ${^*}$-algebra obtained by imposing these 
additional relations by  $\overline{\cA}_o$ \footnote{We note the close relation between the algebra $\overline{\cA}_o$ and the fundemantal group of a complex of groups (cfr. \cite{Hae})}. \smallskip

Now, we want to make the family $\{ \overline{\cA}_o \}$ into a net.
To this end, given $o\leq \tilde o$, we define  
\begin{equation}
\label{Bd:8}
\overline{\jmath}_{\tilde oo} (p,A) := ((\tilde oo)*p,A) \ , \qquad 
(p,A)\in \overline{\cA}_o \ ,  
\end{equation}
and extend it by multiplicativity and linearity to  all of $\overline{\cA}_o$. 
It is easily seen that 
$\overline{\jmath}_{\tilde oo}:\overline{\cA}_o \to \overline{\cA}_{\tilde o}$ 
is a well defined ${^*}$-morphism, which, by homotopy invariance, is invertible. 
Moreover, given $o\leq \tilde o\leq e$ and $(p,A)\in \overline{\cA}_o$, 
by homotopy invariance we have  
\[ 
    \overline{\jmath}_{e\tilde o} \circ \overline{\jmath}_{\tilde oo} (p,A)
  = \overline{\jmath}_{e\tilde o} ((\tilde oo)*p,A) 
 = ( (e\tilde o)*(\tilde oo)*p,A) = ((eo)*p,A) 
 = \overline{\jmath}_{eo} (p,A) \ .
\]
This proves that $(\overline{\cA},\overline{\jmath})_K$ is a net bundle of ${^*}$-algebras.\smallskip

We now introduce a norm making $(\overline{\cA},\overline{\jmath})_K$ a $\rC^*$-net bundle. 
Given $o\in K$, for any  $W\in\overline{\cA}_o$, we define
\begin{equation}
\label{Bd:10}
\norm{W} := \sup_\pi \norm{\pi_o(W)} \ ,
\end{equation}
where the sup is taken over the set of morphisms 
$\pi : (\overline{\cA},\overline{\jmath})_K \to (\cB,\imath)_K$
taking values in $\rC^*$-net bundles. If $\| W \| = 0$ then for every $\pi$ 
and $\tilde o \geq o$ we have
$\| \pi_o(W) \| = 
 \| \imath_{\tilde o o} \circ \pi_o(W) \| =
 \| \pi_{\tilde o} \circ \overline{\jmath}_{\tilde oo}(W) \|$;
thus $\norm{\overline{\jmath}_{\tilde oo}(W)} = 0$ and $\norm{\cdot}$ is well-defined
with respect to the inclusion maps $\overline{\jmath}_{\tilde o o}$. Clearly 
$\|\cdot\|$ is a seminorm for $\overline{\cA}_o$. The completion of the 
quotient of $\overline{\cA}_o$ 
by the ideal of null elements, is a $\rC^*$-algebra that, with an abuse of notation, 
we again denote by $\overline{\cA}_o$. 
This yields  a $\rC^*$-net bundle $(\overline{\cA},\overline{\jmath})_K$ that we call 
the \emph{enveloping net bundle of $(\cA,\jmath)_K$}. 

\begin{proposition}
\label{Bd:11}
Given a net $(\cA,\jmath)_K$ of $\rC^*$-algebras there is a  
unital morphism 
\[
\e : (\cA,\jmath)_K \to (\overline{\cA} , \overline{\jmath})_K \ , 
\] 
satisfying the following properties: 
\begin{itemize}
\item[(i)] 
let $(\varphi,\rh), (\theta,\rh)$ be a pair of morphisms from the enveloping net bundle to 
a $\rC^*$-net bundle, if  
$(\varphi,\rh)\circ \e= (\theta,\rh)\circ \e$, then $ \varphi = \theta$; 
\item[(ii)] for any morphism $(\psi,\rf)$ from $(\cA,\jmath)_K$ into a $\rC^*$-net bundle  $(\cB,\imath)_S$ there is a unique morphism $(\psi^\uparrow,\rf) : (\overline{\cA},\overline{\jmath})_K \to (\cB,\imath)_S$ such that 
            $(\psi,\rf)= (\psi^\uparrow,\rf)\circ \e$.
\end{itemize}
\end{proposition}
\begin{proof}
Given $o\in K$, let 
\begin{equation}
\label{Bd:9}
\e_o(A):= (\iota_o,A) \ , \qquad A\in{\cA_o} \ ,
\end{equation}
where $\iota_o$ is the trivial loop over $o$. 
Properties (\ref{Bd:3},\ref{Bd:4}, \ref{Bd:5})
imply that $\e_o:{\cA_o}\to \overline{\cA}_o$ is a unital ${^*}$-morphism. 
Moreover, for $\tilde o\leq o$ and $A\in\cA_{\tilde o}$ we have 
\[ 
\e_o \circ \jmath_{o\tilde o}(A)  = (\iota_o , \jmath_{o\tilde o}(A)) 
 = (\iota_o*(o\tilde o), A) \\
   = ((o \tilde o), A) 
 = \overline{\jmath}_{o\tilde o}(\iota_{\tilde o}, A) 
  = \overline{\jmath}_{o\tilde o} \circ \e_{\tilde o}(A) \ , 
\]
where isotony and homotopy invariance have been used. Thus the collection 
$\e:=\{\e_o, \ o\in K\}$ is a unital morphism from $(\cA,\jmath)_K$ into 
$(\overline{\cA},\overline{\jmath})_K$. Since, for any morphism $\overline{\pi}$ from $(\overline{\cA},\overline{\jmath})_K$ into 
a $\rC^*$-net bundle, the composition 
$\overline{\pi} \circ \e$ is a morphism from $(\cA,{\jmath})_K$ into a $\rC^*$-net bundle,
$\norm{\e_o(A)} = \sup_{\overline{\pi}} \norm{\overline{\pi}_o \circ \e_o(A)} \leq \sup_{\overline{\pi}}\norm{\overline{\pi}_o(A)}=\norm{A}$, 
so $\e_o$ extends by continuity to  all of $\cA_o$  proving the first part of the statement.\\
\indent $(i)$ Let $(\varphi,\rh), (\theta,\rh): (\overline{\cA},\overline{\jmath})_K\to (\cC,y)_P$ 
be a pair of morphisms as in the statement, where $(\cC,y)_P$ is a $\rC^*$-net bundle. 
Given $o\in K$, let  $p:a\to o$ and $A\in\cA_a$. Using the definition of $\e$,  and  of the 
inclusion maps (\ref{Bd:8}), 
\begin{align*}
 \varphi_o(p,A)& = (\varphi_o\circ \overline{\jmath}_p) (\io_a,A)  
   = (y_{\rh(p)}\circ \varphi_a)(\io_a,A)  \\
   & =  (y_{\rh(p)}\circ \varphi_a\circ \e_a)(A)  = (y_{\rh(p)}\circ \theta_a\circ \e_a)(A)
   =  \theta_o(p,A) \ .   
\end{align*}
So $\varphi_o=\theta_o$ because they coincide on the generators of $\overline{\cA}_o$.\\
\indent $(ii)$ Given a morphism $(\psi,\rf):(\cA,\jmath)_K\to (\cB,\imath)_S$, where  $(\cB,\imath)_S$ is a $\rC^*$-net bundle, define $(\psi^\uparrow,\rf)$ on the generators of $\overline{\cA}_o$ as follows,
\begin{equation}
\label{Bd:12}
\psi^\uparrow_o(p,A):= \imath_{\rf(p)} \circ \psi_a(A) \ , \qquad A\in {\cA_a}\ .
\end{equation}
It easily follows from this definition that  $\psi^\uparrow_o$ preserves isotony and homotopy invariance, and that 
$\psi^\uparrow_o(p,A)\in \cB_{\rf(o)}$ since $\psi_a:\cA_a\to\cB_{\rf(a)}$. 
Extend $\psi^\uparrow_o$ by multiplicativity and linearity to all of $\overline{\cA}_o$.
Note that 
\begin{align*}
 \psi^\uparrow_o \circ \overline{\jmath}_{o\tilde o}(q,A)  &  = 
 \psi^\uparrow_{o}((o\tilde o)*q,A)
  =  \imath_{\rf((o\tilde o)*q)} \circ \psi_a(A) \\
  & = {\imath}_{\rf(o)\rf(\tilde o)} \circ {\imath}_{\rf(q)} \circ \psi_{a}(A)
   = {\imath}_{\rf(o)\rf(\tilde o)} \circ \psi^\uparrow_{\tilde o}(q,A) \ ;
\end{align*}
moreover, $\psi^\uparrow_o\circ \e_o(A)= \psi^\uparrow_o (\iota_o,A) = \psi_o(A)$. 
Uniqueness follows from  $(i)$, completing the proof. 
\end{proof}
In the following we shall refer to the morphism $\e:(\cA,\jmath)_K\to 
(\overline{\cA},\overline{\jmath})_K$ defined by equation (\ref{Bc:9}) as the \emph{canonical embedding} of the net into 
its enveloping net bundle, and to   Prop. \ref{Bd:11} 
as the \emph{universal property} of the enveloping net bundle. This property 
characterizes the enveloping net bundle  up to isomorphism. In particular, it implies 
that any $\rC^*$-net bundle is isomorphic to its enveloping net bundle. \smallskip

%
%
%
%
\indent It is clear from the definition that the enveloping net bundle 
of a net may vanish, i.e., the seminorm (\ref{Bd:10}) may be zero for every $W$.
On these grounds we introduce the following terminology.
\begin{definition}
\label{Bd:14}
We say that a net of $\rC^*$-algebras is \textbf{degenerate} if its enveloping net bundle 
vanishes, and is \textbf{nondegenerate} otherwise. A nondegenerate 
net of $\rC^*$-algebras is \textbf{injective} if the canonical embedding is a monomorphism. 
\end{definition}
Injectivity is the central notion of the present paper. As we shall see in 
\S \ref{Cb}, injectivity for a net of $\rC^*$-algebras 
turns out to be equivalent to the existence of faithful representations. So from now 
on our main task shall be to understand what conditions on a net 
are  necessary or sufficient for injectivity.\\ 
\indent The next result shows that the enveloping net bundle is the 
object that we were looking for: 
an object uniquely associated with a net of $\rC^*$-algebras which takes into account 
the topology of the poset and reduces, in the simply connected case, 
to the universal algebra defined by Fredenhagen. 
\begin{lemma}
\label{Bd:17}
Given a net $(\cA,{\jmath})_K$.  If $K$ is simply connected, then there is a canonical isomorphism 
            \[
            \rho : (\overline{\cA},\overline{\jmath})_K \to (\cA^{t},\jmath^{t})_K
            \ ,
            \] 
            where $(\cA^{t},\jmath^{t})_K$ is the trivial net with fibre 
            Fredenhagen's universal $\rC^*$-algebra $\cA^\ru$. 
            In particular, 
            if $K$ is upward directed $\cA^u$ is isomorphic to the $\rC^*$-inductive limit 
            of $(\cA,{\jmath})_K$. 
\end{lemma}
\begin{proof}
$K$ being simply connected, $\rho$ is defined at the $^*$-algebraic
level observing that: 
(1) the generators of $\overline{\cA}_o$ are the same as the generators 
of $\cA^\ru$, since $(p,A)=(q,A)$ for any pair of paths $p,q:a\to o$;
(2) the relations (\ref{Bd:3})-(\ref{Bd:6}) are the same as those 
defining $\cA^\ru$ (see \cite{Fre}).
\end{proof}
%
%
%
%
%
%
%

\begin{remark}
\label{Bd:18}
 Let $G$ be a symmetry group for $K$ and $(\cA,\jmath)_K$ a $G$-covariant net 
(see \S \ref{Ba}). Define, for any  $p:a\to o$ and  $A\in\cA_a$,
\begin{equation}
\label{net:env:14a}
\overline{\alpha}_o^g (p,A) := (gp,\alpha^g_a(A)) \ , \qquad  g\in G
\ .
\end{equation}
%
Then it is easily seen that $\overline{\alpha}$ yields an action on 
$(\overline{\cA},\overline{\jmath})_K$, making it
$G$-covariant, and that $\e$ is a 
$G$-covariant morphism. 
Furthermore, assume that $G$ is a continuous symmetry group
of $K$ (in the sense of (\ref{Ab:1})), and that the action $\alpha$ of $G$ on $(\cA,\jmath)_K$ 
is continuous (see \S \ref{Ba}). The action $\overline{\alpha}$ of $G$ on 
the enveloping net bundle is continuous too. 
It is enough to prove this on the generators of the fibres of the enveloping net bundle. 
Consider a net $\{g_\lambda\}$ converging to the identity of the group $G$. 
Given $o$,  let $\tilde o >  o$. For any $p:a\to o$ and $A\in\cA_a$ we have,
%
%
using  
$(p,A) = \overline{\jmath}_p \circ \e_a (A)$
and (\ref{net:env:14a}),
\begin{align*}
\| \overline{\jmath}_{\hat o\, g_\lambda o}\circ \overline{\alpha}^{g_\lambda}_o(p,A)- 
   \overline{\jmath}_{\hat o\, o}\circ (p,A) \| =& 
  \| \overline{\jmath}_{\hat o\, g_\lambda o}\circ \overline{\alpha}^{g_\lambda}_o\circ \overline{\jmath}_p \circ \e_a (A) - 
   \overline{\jmath}_{\hat o\, o}\circ \overline{\jmath}_p \circ \e_a (A) \| \\
   = &   \| \overline{\jmath}_{\hat o\, g_\lambda o}\circ \overline{\jmath}_{gp}\circ 
    \overline{\alpha}^{g_\lambda}_a\circ \e_a (A) - 
   \overline{\jmath}_{\hat o\, o}\circ \overline{\jmath}_p \circ \e_a (A) \| \\
   = &   \| \overline{\jmath}_{\hat o\, g_\lambda o}\circ \overline{\jmath}_{gp}\circ 
   \e_{g_\lambda a}\circ \alpha^{g_\lambda}_a\ (A) - 
   \overline{\jmath}_{\hat o\, o}\circ \overline{\jmath}_p \circ \e_a (A) \| \ . 
\end{align*}
Now, take $\hat a >a$. Since $g_\lambda a$ are eventually smaller than  $\hat a$,  
using the above  relation we have 
%
%
%
\begin{align*}
\| \overline{\jmath}_{\hat o\, g_\lambda o}\circ & \overline{\alpha}^{g_\lambda}_o(p,A)- 
    \overline{\jmath}_{\hat o\, o}\circ (p,A) \| = \\    
  & =  \| \overline{\jmath}_{\hat o\, g_\lambda o}\circ \overline{\jmath}_{gp}\circ 
   \e_{g_\lambda a}\circ \alpha^{g_\lambda}_a\ (A) - 
   \overline{\jmath}_{\hat o\, o}\circ \overline{\jmath}_p \circ \e_a (A) \| \\
   & = 
     \| \overline{\jmath}_{(\hat o, g_\lambda o)* gp * \overline{(\hat a,g_\lambda a)}}\circ 
               \overline{\jmath}_{\hat a \, g_\lambda a}\circ  
   \e_{g_\lambda a}\circ \alpha^{g_\lambda}_a\ (A) - 
   \overline{\jmath}_{(\hat o, o)*p* \overline{(\hat a,a)}}\circ  \overline{\jmath}_{\hat a\, a}\circ \e_a (A) \| \\
     & = 
     \| \overline{\jmath}_{(\hat o, g_\lambda o)* gp * \overline{(\hat a,g_\lambda a)}}\circ 
               \e_{\hat a} \circ  \jmath_{\hat a \, g_\lambda a}\circ  \alpha^{g_\lambda}_a\ (A) - 
   \overline{\jmath}_{(\hat o, o)*p* \overline{(\hat a,a)}}\circ  \e_{\hat a} \circ \jmath_{\hat a\, a}(A) \| \ .
\end{align*}
By continuity of the $G$-action on $K$ (see (\ref{Ab:1})), the paths 
$(\hat o, g_\lambda o)* gp * \overline{(\hat a,g_\lambda a)}$ and 
$(\hat o, o)*p* \overline{(\hat a,a)}$ are homotopic.  So that 
\begin{align*}
\| \overline{\jmath}_{\hat o\, g_\lambda o}\circ & \overline{\alpha}^{g_\lambda}_o(p,A)- 
    \overline{\jmath}_{\hat o\, o}\circ (p,A) \| = \\
&  =     \| \overline{\jmath}_{(\hat o, o)*p* \overline{(\hat a,a)}}\circ 
               \e_{\hat a} \circ  \jmath_{\hat a \, g_\lambda a}\circ  \alpha^{g_\lambda}_a\ (A) - 
   \overline{\jmath}_{(\hat o, o)*p* \overline{(\hat a,a)}}\circ  \e_{\hat a} \circ \jmath_{\hat a\, a}(A) \| \\
&  \leq     \|  \jmath_{\hat a \, g_\lambda a}\circ  \alpha^{g_\lambda}_a\ (A) - \jmath_{\hat a\, a}(A) \| \ ,    
\end{align*}
which goes to zero as $g_\lambda\to e$ because of the continuity of $\alpha$. 
This proves that $\overline{\alpha}$ is continuous.
\end{remark}
We have just seen that if a net is covariant then the enveloping net bundle is covariant too. 
A property of a net not in general inherited by the enveloping 
net bundle is causality.  This is clear from the definition,  
since all fibres of the net are involved in  defining 
a single fibre of the enveloping net bundle.\\
\indent We conclude by showing the stability of injectivity  under 
morphisms faithful on the fibres (see \S \ref{Ba}), and  the functoriality 
of the enveloping net bundle. 

\begin{proposition}
\label{Bd:19}
The following assertions hold. 
\begin{itemize}
\item[(i)] To any morphism $(\pi,\rf):(\cA,\jmath)_K\to (\cB,\imath)_P$ 
there corresponds a morphism $(\overline{\pi},\rf):(\overline{\cA},\overline{\jmath})_K\to (\overline{\cB},\overline{\imath})_P$ satisfying 
\begin{equation}
\label{Bd:19a}
 (\overline{\pi},\rf)\circ \e = \tilde \e \circ (\pi,\rf) \ , 
\end{equation}
where $\e$ and $\tilde e$ are, respectively, the canonical embeddings of 
the nets  $(\cA,\jmath)_K$ and $(\cB,\imath)_P$ into the corresponding enveloping net bundles.
If  $(\pi,\rf)$ is faithful on the fibres and $(\cB,\imath)_P$ is injective, 
then $(\cA,\jmath)_K$ is injective too. 
\item[(ii)] Assigning the  enveloping net bundle yields a functor 
from the category of net of $\rC^*$-algebras to the category 
of $\rC^*$-net bundles.
\end{itemize}
\end{proposition}

\begin{proof}
$(i)$ Given $o\in K$, define
\begin{equation}
\label{Bd:20}
 (\overline{\pi},\rf)_o(p,A) := (\rf(p),\pi_a(A))  
 \ , \qquad 
 (p,A) \in \overline{\cA}_o \ . 
\end{equation}
Clearly $(\overline{\pi},\rf)_o(p,A)\in \overline{\cB}_{\rf(o)}$. 
To prove that (\ref{Bd:20}) is well defined, we consider $a \leq \tilde a$
and compute
\begin{align*}
(\overline{\pi},\rf)_o (p*(\tilde aa),A) & 
 = (\rf(p*(\tilde aa)),\pi_a(A)) 
  = (\rf(p)*\rf(\tilde aa),\pi_a(A)) \\
& = (\rf(p),\imath_{\rf(\tilde a) \rf(a)} \circ \pi_a(A))
  = (\rf(p),\pi_{\tilde a} \circ \jmath_{\tilde aa} (A) ) \\
& = (\overline{\pi},\rf)_o (p,{\jmath}_{\tilde aa}(A)) \ ; 
\end{align*}
thus (\ref{Bd:20}) is well-defined at the level of isotony.
Passing to homotopy invariance, we note that if $p,q:a\to o$ are homotopic
then $\rf(p)$ is homotopic to $ \rf(q)$, and we have 
$(\rf(p),\pi_a(A)) = (\rf(q),\pi_a(A))$.
This proves that (\ref{Bd:20}) is well posed.  
If $o\leq \tilde o$, then by homotopy invariance
\begin{align*}
 \overline{\imath}_{\rf(\tilde o) \rf(o)} \circ (\overline{\pi},\rf)_o  (p,A) 
& = ( (\rf(\tilde o) \rf(o))*\rf(p),\pi_a(A) ) 
 =  ( \rf_*(\tilde oo)*p , \pi_a(A) ) \\
& = (\overline{\pi},\rf)_{\tilde o} ((\tilde oo)*p,A)
 =  (\overline{\pi},\rf)_{\tilde o} \circ \overline{\jmath}_{\tilde oo} (p,A) \ . 
\end{align*}
Observing that  
\[
(\overline{\pi},\rf)_o\circ \e_o(A) = (\overline{\pi},\rf)_o (\io_o,A)= 
(\io_{\rf(o)}, \pi_{o}(A)) = 
\tilde\e_{\rf(o)}(\pi_{o}(A))=  \tilde\e_{\rf(o)}\circ (\pi,\rf)_o(A)  \ , 
\]
for any $o$ and $A\in\cA_o$, equation (\ref{Bd:19a}) follows. Finally, 
if  $(\pi,\rf)$ is faithful on the fibres and $(\cB,\imath)_P$ is injective, 
the r.h.s. of the above equation is the  composition 
of faithful morphisms for any $o$; so,  $\e$ is a monomorphism. $(ii)$ Clearly 
$(\overline{\pi'\circ\pi},\rf'\circ\rf)  =
  (\overline{\pi'},\rf') \circ (\overline{\pi},\rf)$,
and this concludes the proof.
\end{proof}
This proposition and Prop.\ref{Bbc:1} imply that 
the net  $(\rC^*(\Lambda^l),\rC^*(\lambda))_K$,  given  in  \S \ref{Bbc},  
is injective.  



\section{States and representations}
\label{C}
We study states and representations of nets of $\rC^*$-algebras providing  
sufficient conditions for the existence of states (representations), and  of 
invariant  states (covariant representations) when the poset 
is endowed with a symmetry group. Furthermore, we show that for a net of $\rC^*$-algebras,  
injectivity is equivalent to the existence of faithful representations. 
Aspects of the decomposition of representations are studied in Appendix \ref{AppRep}.

%
%
\subsection{States}
\label{Ca}
We relate states of a $\rC^*$-net bundle to those of the corresponding 
holonomy dynamical system. This will allow us to prove that, when the homotopy group 
of the underlying poset is amenable, any nondegenerate net has states, which 
are invariant whenever the net is  covariant under an amenable 
symmetry group.\bigskip 

A \emph{state}
of a net of $\rC^*$-algebras $(\cA,{\jmath})_K$  is a family 
$\omega := \{ \omega_o \ , o\in K \}$, 
where $\omega_o$ is a state of the $\rC^*$-algebra $\cA_o$, fulfilling the relation
\begin{equation}
\label{Ca:1}
\omega_o = \omega_{a} \circ {\jmath}_{ao}
\ \ , \ \ 
o \leq a
\ .
\end{equation}
We shall denote the set of states of $(\cA,{\jmath})_K$ by 
$\cS(\cA,{\jmath})_K$. \bigskip

It is easily seen that if $(\phi,\rf):(\cB,\imath)_P\to (\cA,{\jmath})_K$ is a unital morphism 
and $\omega$ is a  state of $(\cA,{\jmath})_K$, then the composition $\omega\circ\phi$ defined by 
\begin{equation}
\label{Ca:2}
(\omega\circ\phi)_o = \omega_{\rf(o)} \circ \phi_o \ , \qquad o\in P \ , 
\end{equation}
yields a state of the net  $(\cB,\imath)_P$. Another property easy to verify 
is the following. When $(\cA,{\jmath})_K$ is a $\rC^*$-net bundle then, by (\ref{Ca:1}) 
\begin{equation}
\label{Ca:3}
\omega_a = \omega_{o} \circ {\jmath}_{p} \ , \qquad  p:a\to o \ . 
\end{equation}
Our first result  relates states of a $\rC^*$-net bundle to 
invariant states of the corresponding holonomy dynamical system. 
\begin{lemma}
\label{Ca:4}
The set of states of a $\rC^*$-net bundle is in one-to-one correspondence with 
the set of invariant states of the associated holonomy dynamical system.  \end{lemma}
\begin{proof}
Consider the holonomy dynamical system $(\cA_*,\pi^o_1(K),\jmath_*)$  defined 
with respect to $o\in K$ (\ref{Bc:4}).
Let $\omega$ be a state of a $\rC^*$-net bundle $(\cA,{\jmath})_K$.  Then  $\omega_*:=\omega_o$ is an invariant state of $(\cA_*,\pi^o_1(K),\jmath_*)$. 
In fact by (\ref{Ca:3}) we have that 
\[
\omega_*\circ \jmath_{*,[p]}= \omega_o\circ \jmath_{*,[p]}=  \omega_o \circ {\jmath}_p = \omega_o =\omega_*
\ , 
\]
for any $[p]\in\pi_1^o(K)$. 
Conversely, let $\varphi$ be an invariant state 
of $(\cA_*,\pi_1^o(K),\jmath_*)$. Take a path frame $P_o:=\{p_{(a,o)}, \ a\in K\}$ and define a state on the associated net bundle
$(\cA_{**},\jmath_{**})_K$ (see \S \ref{Bc}) by 
\[
\varphi_{*,a}:= 
\varphi\circ \jmath_{**,p} \stackrel{(\ref{Bc:7})}{=} 
\varphi \circ \jmath_{*, [p_{(o, o)}*p*p_{(a,o)}] } = \varphi \circ \jmath_{*, [p*p_{(a, o)}]}
\ , \qquad a\in K \ , 
\] 
for some path $p:a\to o$, where the fact that $p_{(o,o)}$ is homotopic to the trivial loop
over $o$ has been used.  Since $\varphi$ is $\jmath_*$-invariant for every $q : a\to o$ we have
$\varphi_{*,a}= \varphi\circ \jmath_{**,p}= 
\varphi\circ \jmath_{**,p*\overline{q}}\circ \jmath_{**,q} = \varphi\circ\jmath_{**,q}$,
so that the family $\varphi_*:=\{ \varphi_{*,a} \}$ is well-defined (note in particular 
that $\varphi_{*,o}=\varphi$). For the same reason we have 
\[
\varphi_{*,a} \circ \jmath_{**,a\tilde a} =
\varphi\circ \jmath_{**,p}\circ \jmath_{**,a\tilde a} = 
\varphi\circ \jmath_{**,p*(a,\tilde a)} = \varphi_{*,\tilde a}
\]
for any  $\tilde a \leq a$, thus $\varphi_* \in \cS(\cA_{**},\jmath_{**})_K$.
Composing $\varphi_*$ with the isomorphism 
$\tau:(\cA,\jmath)_K\to (\cA_{**},\jmath_{**})_K$
defined by the equation (\ref{Bc:9}) yields a state of $(\cA,\jmath)_K$. \smallskip  

Finally,  we prove that these mappings are the inverse of one another. Given 
a state $\omega$ of $(\cA,\jmath)_K$,  we have 
\[
(\omega_{**}\circ \tau)_a = \omega_{**,a}\circ \tau_a =
\omega_{*} \circ \jmath_{**,p} \circ \tau_a =
\omega_{*} \circ \tau_o \circ  \jmath_{p} = 
\omega_o \circ \jmath_{p} = \omega_a \ , 
\]  
for some path $p:a\to o$, where we have used  the fact that $\tau_o=\mathrm{id}_o$ (see definition 
\ref{Bc:9}) and the equation (\ref{Ca:3}). Conversely, if $\varphi$ is a state  of the holonomy dynamical system, then  
$(\varphi_{*}\circ\tau)_* = (\varphi_{*}\circ\tau)_o = 
 \varphi_{*,o}\circ\tau_o = \varphi_{*,o}$, 
because, as observed above,  $\varphi_{*,o}=\varphi$ completing the proof. 
\end{proof}
We now are ready to give the main result on the existence of states for nets of $\rC^*$-algebras. 
\begin{proposition}
\label{Ca:5}
Let $K$ be a poset with amenable homotopy group. 
Then any nondegenerate net of $\rC^*$-algebras over $K$ has states.
\end{proposition}
\begin{proof}
Since a nondegenerate net as a nonvanishing enveloping net bundle 
it is enough, by (\ref{Ca:2}), to prove the statement when  
$(\cA,{\jmath})_K$ is a $\rC^*$-net bundle. This follows by the previous lemma, 
since any $\rC^*$-dynamical system 
with an amenable group has invariant states.  
%
%
\end{proof}

Let now $(\cA,\jmath,\alpha)_K$ be a $G$-covariant net. A state $\varphi \in \cS(\cA,\jmath)_K$
is said to be $G$-\emph{invariant} whenever
\[
\varphi_{go} \circ \alpha^g_o := \varphi_o
\ \ , \ \
\forall o \in K
\ , \
g \in G
\ .
\]
The next result gives conditions for the existence of $G$-invariant states.
\begin{proposition}
\label{Ca:6} 
Let $G$ be an amenable group. Then the following assertions hold:
\begin{itemize}
\item[(i)]  Any $G$-covariant $\rC^*$-net bundle having states has $G$-invariant states. 
\item[(ii)] If the fundamental group of $K$ is amenable, then any nondegenerate $G$-covariant net 
over $K$ has $G$-invariant states.
\end{itemize}
\end{proposition}
\begin{proof}
$(i)$ Let $(\cA,\jmath,\alpha)_K$ be a $G$-covariant $\rC^*$-net bundle
and $\omega \in \cS(\cA,\jmath)_K$. For any $o\in K$ and $A \in \cA_o$, define 
\[ 
f_o^A(g):= \omega_{go} \circ \alpha^g_o(A) \ , \qquad g \in G \ .
\] 
It is clear that $f_o^A\in L^\infty(G)$ for any $o\in K$ and $A\in\cA_o$. Moreover the mapping 
${\cA_o}\ni A\to\rf_o^A\in L^\infty(G)$ is linear and positive. We also note that 
the following two relations hold: 
\[
\begin{array}{llc}
f_a^{{\jmath}_{ao}(A)}= f_o^{A} \ , & o\leq a \ ,  & \qquad \qquad     (*) \\[5pt] 
f_{ho}^{\alpha^h_o(A)} = (f_o^A)_h \ , & h\in G \ , & \qquad \qquad (**)
\end{array}
\]
where $(f_o^A)_h$ means the right translation of $f_o^A$ by $h$. 
In fact, for any $A\in{\cA_o}$ and $g\in G$ we have 
\[
f_a^{{\jmath}_{ao}(A)}(g)  = \omega_{ga}(\alpha^g_a({\jmath}_{ao}(A))  
				  = \omega_{ga}({\jmath}_{ga\,go}(\alpha^g_o(A))
				=  \omega_{go}(\alpha^g_o(A)) 
				 =f_o^{A}(g) \ , 
\]
proving  $(*)$. 
Moreover, for any $A\in{\cA_o}$ and $g,h\in G$ we have 
\[
f_{ho}^{\alpha^h_o(A)}(g)   = 
\omega_{gho}(\alpha^{g}_{ho}(\alpha^h_o(A)) 
  = \omega_{gho}(\alpha^{gh}_o(A) ) 
  = f_o^A(gh)
 = (f_o^A)_h(g) \ , 
\]
%
proving $(**)$. Now, 
let $\mu$ be a right invariant mean over $L^\infty(G)$. Define 
\[
\varphi_o(A):= \mu(f_o^A) \ , \qquad A\in{\cA_o} \ ;
\]
this is a state over ${\cA_o}$. 
Moreover, the collection $\varphi:=\{\varphi_o \ ,  \ o\in K\}$ is a 
$G$-invariant state of $(\cA,\jmath)_K$, in fact
by the relation $(*)$  
$\varphi_o \circ {\jmath}_{oa}(A)=
 \mu(f_o^{\jmath_{oa}(A)}) =
 \mu(f_o^A)=
 \varphi_o(A)$. 
On the other hand, by the relation $(**)$ it follows that 
\[
\varphi_{ho} \circ \alpha^h_o(A) = 
\mu(f_{ho}^{\alpha^h_o(A)}) = 
\mu((f_o^A)_{h}) =
\mu(f_o^A) = 
\varphi_o(A) \ , 
\]
where the invariance of $\mu$ has been used.\\ 
\indent $(ii)$ Since $\pi_1^o(K)$ is amenable, by Proposition \ref{Ca:5} 
the enveloping  net bundle $(\overline{\cA},\overline{\jmath})_K$ has states.
Applying Rem.\ref{Bd:18} we conclude that $(\overline{\cA},\overline{\jmath})_K$ is
$G$-covariant with an action $\overline{\alpha}$ satisfying
\[
\overline{\alpha}^{g}_o \circ \e_o = \e_{go} \circ \alpha^g_o
\ \ , \ \
\forall o \in K
\ , \
g \in G
\ .
\]
Applying $(i)$, we conclude that $(\overline{\cA},\overline{\jmath},\overline{\alpha})_K$ has
an invariant state $\overline{\varphi}$, thus defining
$\varphi_o := \overline{\varphi}_o \circ \e_o$, for any 
$o \in K$,
yields an invariant state of $(\cA,\jmath)_K$.
\end{proof}

\subsection{Representations}
\label{Cb}
We now study representations of nets of $\rC^*$-algebras. 
We relate representations of a $\rC^*$-net bundle to those of the corresponding holonomy 
dynamical system, and representations of a net to those of the enveloping net bundle.
This leads to the equivalence between injectivity and existence of faithful representations. 
Injective nets defined over a poset with amenable fundamental group, and  an amenable 
symmetry group, have covariant representations. We also characterize  
those nets having Hilbert space representations and those nets having a trivial enveloping net bundle.\bigskip

Let $(\cA,{\jmath})_K$ be a net of $\rC^*$-algebras. 
A \emph{representation} of $(\cA,{\jmath})_K$  
is a pair $(\pi,U)$, where $\pi$ is a family of Hilbert space representations
$\pi_o : {\cA_o} \to \mathfrak{B}(\cH_{o})$, $o \in K$  and   
$U$ is a family of unitaries     
$U_{ao} : \cB(\cH_o)  \to \cB(\cH_{a})$, $o \leq a$, called \emph{inclusion operators},  
such that    
\begin{equation}
\label{Cb:1}
U_{ao}\in (\pi_o,\pi_{a}\circ {\jmath}_{ao})  \ , \qquad o\leq a  \ , 
\end{equation}
and 
\begin{equation}
\label{Cb:2}
U_{ea}\, U_{ao} = U_{eo} \ , \qquad o\leq a \leq o \ .
\end{equation}
The representation 
$(\pi,U)$ is said to be \emph{faithful} if $\pi_o$ is a faithful representation 
of ${\cA_o}$ for any $o\in K$. A \emph{Hilbert space representation} of $(\cA,\jmath)_K$ is a
representation of the form $(\pi,1)$ (here we assume that every 
Hilbert space $\cH_o$, $o \in K$, coincides with a fixed 
Hilbert space $\cH$ whose identity is  denoted by $1$).

\begin{remark}
\label{Cb:3}
In the context of the algebraic quantum field theory 
a representation of a net of $\rC^*$-algebras  usually means what we call a Hilbert 
space representation. The first time  a representation, in the sense of the present paper,   
appeared was  in \cite{FH};  the reconstruction of a state of an algebra associated to a region 
of a spacetime from a family of states of its subregions yielded a collection of representations 
and  unitary operators  satisfying the above relations. This structure  has been promoted 
to the r\^ole of a representation of a net of $\rC^*$-algebras in \cite{BR} 
(and called a unitary net representation) where 
its topological content was analyzed and where, in particular,  
its r\^ole in the description of charges induced by the topology of a spacetime 
was pointed out. 
\end{remark}

With the above notation, it is easily seen that $U_{oo}=1_o$ for any $o\in K$. 
Since unitaries  $U$ are invertible, we define $U_{oa}:= {U}^*_{ao}$ for any $o\leq a$.   
Note that the pair $(\cH,U)_K$, $\cH := \{ \cH_o \}$, defines a Hilbert net bundle,
and (\ref{Cb:1}), (\ref{Cb:2}) imply that $(\pi,U)$ yields a morphism 
\begin{equation}
\label{Cb:4}
\pi : (\cA,\jmath)_K \to (\cB \cH , \ad U)_K \ ,
\end{equation}
where $(\cB \cH , \ad U)_K$ is the net bundle with fibres $\cB(\cH_o)$, $o \in K$, 
with inclusion maps $\ad U_{ao}$, $o \leq a$, defined by the adjoint action.  
So a representation can be seen as a morphism from the given net 
to the $\rC^*$-net bundle defined by a Hilbert net bundle. Both the concrete
$\rC^*$-net bundle and the Hilbert net bundle are trivial in the case of Hilbert space 
representations. \smallskip

An \emph{intertwiner}  between two representations $(\pi,U)$ and $(\pi',U')$ of
$(\cA,\jmath)_K$ is a family of bounded linear operators 
$T:=\{T_{o}:\cH_o\to\cH'_o \ , \ o\in K\}$ such that $T_o\in(\pi_o,\pi'_o)$,  and 
\begin{equation}
\label{Cb:5}
 T_o\, U_{oa} = U'_{oa}\, T_a \ , \qquad  a\leq o \ .  
\end{equation}
When all $T_o$ are unitaries we shall say that $T$ is a \emph{unitary} intertwiner. 
Two representations $(\pi,U)$ and $(\pi',U')$ are \emph{equivalent}
if they have a unitary intertwiner. A representation is said to be 
\emph{topologically trivial} whenever it is equivalent to a Hilbert space representation
(the motivation of this terminology will soon be clear). Finally, a representation $(\pi,U)$ is said to \emph{vanish} whenever $\pi_a(A)=0$ for any $a\in K$  and $A\in\cA_a$.\smallskip

We can always assume that a representation $(\pi,U)$ is defined on a fixed Hilbert space
(see \cite{BR} for details). For any 1-simplex $b$ 
define $U_b:= U_{\partial_0b|b|}\, U_{|b|\partial_1b}$, so we have 
a unitary $U_b:\cH_{\partial_1b}\to \cH_{\partial_0b}$. 
Extend $U$ from 1-simplices to paths 
in the usual way,
\begin{equation}
\label{Cb:6}
U_p := U_{b_n} \, \cdots \, U_{b_2}\,  U_{b_1}
\  , \qquad 
p = b_n* \cdots b_2* b_1
\ .
\end{equation}
Afterwards, fix a path frame $P_o$ and define  
${\pi}'_a(\cdot):= U_{p_{(o,a)}}\, \pi_a(\cdot) \, U_{p_{(a,o)}}$ with  $a\in K$ 
(note that ${\pi}'_o= \pi_o$) and 
${U}'_{ae}:= U_{p_{(o,a)}}\, U_{ae}\,U_{p_{(e,o)}}$ with  $a\leq e$. Then  
the pair $(\pi',U')$ defines a representation of the net $(\cA,\jmath)$ into a fixed 
Hilbert space; the family $T_a:=U_{p_{(o,a)}}$, with $a\in K$,  is a unitary intertwiner from 
$(\pi,U)$ to $(\pi',U')$. On these grounds, unless otherwise stated,  
we assume from now on that \emph{all representations are defined on a fixed Hilbert space}.
\smallskip

The topological content of a representation $(\pi,U)$ can be easily seen 
by considering the associated Hilbert net bundle $(\cH,U)$. 
The same reasoning used in  defining the holonomy dynamical system (\S \ref{Bc}), yields 
a group morphism 
\begin{equation}
\label{eq.u*}
U_* : \pi^o_1(K) \to \cU(\cH)
\end{equation}
where $\cU(\cH)$ is the group of unitary operators of the Hilbert space $\cH$. 
Explicitly, one extends $U$ to simplices as above and observes that the mapping
$U:\Si_1(K)\to U_b\in\cB(\cH)$ satisfies  the 1-cocycle relation
$U_{\partial_0c}\,U_{\partial_2c}=U_{\partial_1c}$ for any $2$-simplex $c$. 
So $U$, in turn, defines a representation $U_*$ of the fundamental group of $K$. 
We shall call $U_*$ \emph{the holonomy representation} associated with $(\pi,U)$.\smallskip

We list the following results from \cite{BR}, which can be proved  as in
\S \ref{Bc}:
\begin{itemize}
\item[(i)] If two representations are equivalent then the 
corresponding representations of $\pi_1(K)$ are equivalent. 
\item[(ii)] If $K$ is simply connected then any representation 
          is equivalent to a Hilbert space representation. 
\end{itemize}
These two points explain the term 'topologically trivial representation'. \smallskip 

The first task is to relate representations of a net to those of the enveloping net bundle.  
\begin{lemma}
\label{Cb:7}
Any representation of the net $(\cA,\jmath)_K$ extends uniquely to a representation of 
$(\overline{\cA},\overline{\jmath})_K$, and this yields a bijective correspondence 
between representations of $(\cA,\jmath)_K$ and representations of $(\overline{\cA},\overline{\jmath})_K$.
\end{lemma}
\begin{proof}
Let $\e$ be the canonical embedding of the net into the enveloping net bundle. 
For any representation $(\si,V)$ of the enveloping net bundle, the pair 
$(\si\circ\e,V)$ defines a representation of the net. Conversely, 
by (\ref{Cb:4}), any  representation $(\pi,U)$ of $(\cA,\jmath)_K$ defines 
a morphism $\pi : (\cA,\jmath)_K \to (\cB \cH , \ad U)$. By Prop. \ref{Bd:11} 
there is a unique morphism 
$\pi^\uparrow:(\overline{\cA},\overline{\jmath})_K \to (\cB \cH , \ad U)$, defined by 
equation (\ref{Bd:12}),  
such that $\pi^\uparrow\circ \e= \pi$.  
This completes the proof. 
%
\end{proof}
We note that the extension of a faithful representation of a net 
to the enveloping net bundle  need not be faithful.\\ 
\indent The next result relates representations of $\rC^*$-net bundles to covariant representations
of the associated dynamical system. 
\begin{lemma}
\label{Cb:8}
Representations of a $\rC^*$-net bundle are, up to equivalence, in bijective correspondence 
with covariant representations of the corresponding holonomy dynamical system. 
\end{lemma}
\begin{proof}
We give a sketch of the proof since the reasoning  is similar to that of the proof of Lemma \ref{Ca:4}. By (\ref{Cb:4}), we have that every representation $(\pi,U)$ of $(\cA,\jmath)_K$ defines 
a morphism $\pi : (\cA,\jmath)_K \to (\cB \cH , \ad U)$. Thus, by Prop.\ref{Bc:10}, we have 
the morphism of dynamical systems
\begin{equation}
\label{Cb:9}
\pi_* : \cA_* \to \cB(\cH)
\ , \
\pi_* \circ \jmath_{*,[p]} = \ad U_{*,[p]} \circ \pi_*
\ \ , \ \
\forall [p] \in \pi_1^o(K)
\ ,
\end{equation}
i.e., $\pi_*$ is a covariant representation of $( \cA_*,\pi_1^o(K),\jmath_* )$.
Conversely, given  a covariant representation $(\eta,V)$ of 
$( \cA_*,\pi_1^o(K),\jmath_* )$ on the Hilbert space $\cH$ then, 
again by Prop.\ref{Bc:10}, we can define the 
net bundle $( \cB(\cH)_* , \ad V_* )_K$ and the morphism
\[
\eta_* : (\cA_{**},\jmath_{**})_K \to ( \cB(\cH)_* , \ad V_* )_K
\ .
\]
Since $(\cA_{**},\jmath_{**})_K$ is isomorphic to $(\cA,\jmath)_K$, composing
with $\eta_*$ gives  the desired representation.
\end{proof}
After this the relation between injectivity and the existence of faithful representations 
follows easily. 
\begin{theorem}
\label{Cb:10}
The following assertions hold. 
\begin{itemize}
\item[(i)] A net of $\rC^*$-algebras is injective if, and only if, it has  faithful representations. 
\item[(ii)] A net of $\rC^*$-algebras is nondegenerate if, and only if, it has  nonvanishing
representations.
\end{itemize} 
\end{theorem}
\begin{proof}
$(i)$ ($\Leftarrow$)  follows by Prop.\ref{Bd:19}.$i$. 
($\Rightarrow$) By the previous lemma it suffices 
to show the existence of faithful covariant representations for the dynamical system associated with the enveloping net bundle. But this is true for any dynamical system,  
completing the proof of $(i)$. A similar reasoning lead to the proof of $(ii)$. 
\end{proof}
We discuss the existence of Hilbert space representations. Let $(\rA,G,\alpha)$ be 
a $\rC^*$-dynamical system. A $G$-\emph{invariant representation} is a representation $\pi$ of 
$\rA$ such that $\pi\circ \alpha_g=\pi$ for any $g\in G$. 
%
%
Let $J^G_\rA$ the $G$-invariant ideal of $\rA$ generated by the elements $A-\alpha_g(A)$ for 
$A\in\rA$ and $g\in G$; then we have the following result:
\begin{lemma}
\label{Cb:10aa}
A $\rC^*$-dynamical system $(\rA,G,\alpha)$  has  nontrivial $G$-invariant representations if, 
and only if, the ideal $J^G_\rA$ is proper. 
\end{lemma}
\begin{proof}
$(\Rightarrow)$ If $\pi$ is an invariant representation, then $J^G_{\rA}$ lays within the $Ker(\pi)$.
$(\Leftarrow)$ If $J^G_\rA$ is proper, then the quotient 
$\widehat{\rA}:=\rA/ J^G_\rA$ is not trivial, and since  $J^G_{\rA}$ is 
$G$-invariant, the action  $\alpha$ lifts to an action $\widehat{\alpha}$ on $\widehat{\rA}$. 
However this action is trivial according to the definition of $J^G_\rA$. So take  
a faithful representation $\si$ of $\widehat{\rA}$ and define 
$\pi(A):=\si(\widehat{A})$ for any $A\in\rA$. Then $\pi\circ\alpha_g(A)=\si(\widehat{\alpha_g(A)})=
\si(\widehat{A})=\pi(A)$ for any $A\in\rA$ and $g\in G$, completing the proof. 
\end{proof}
We are ready to provide a necessary and sufficient condition for the existence of Hilbert space representations. 
\begin{proposition}  
\label{Cb:10ab}
An injective net of $\rC^*$-algebras $(\cA,\jmath)_K$ has nontrivial Hilbert space representations 
if, and only if, the ideal ${J}^{\pi^o_1(K)}_{\overline{\cA}_*}$ is proper, where
$\overline{\cA}_*$ is the holonomy dynamical system of the enveloping net bundle. 
\end{proposition} 
\begin{proof}
Follows straighforwardly from the previous lemma and from Lemma \ref{Cb:8}
\end{proof}
Examples of injective nets where this ideal fails to be proper will be given in \S \ref{Daa}.\\
\indent We now want to characterize those nets having a trivial enveloping net bundle. To this end, 
let us introduce the following notion. A representation $(\pi,U)$ of a net $(\cA,\jmath)_K$ 
is said to be \emph{quasi (topologically) trivial} whenever for any $o\in K$ the relation 
\begin{equation}
\label{Cb:10a}
U_p\in (\pi_o,\pi_o) \ , \qquad p:o\to o \ , 
\end{equation}
holds. This  means that the coupling between 
the analytical and the topological content of a quasi trivial representation is 'artificial', 
and can be completely removed from this representation. More precisely, let $(\pi,U)$ be 
a quasi trivial representation. Take a path frame $P_o=\{p_{(a,o)}, \ a\in K\}$, 
and define 
\[
 \si_a:= \ad U_{p_{(o,a)}}\circ \pi_a  \ , \qquad a\in K \ , 
\]
Then $(\si,\mathbbm{1})$ is a Hilbert space representation of the net. In fact 
given $a \leq \tilde a$, by (\ref{Cb:10a}) we have 
\begin{align*}
\si_{\tilde a}\circ \jmath_{\tilde aa} & = \ad U_{p_{(o,\tilde a)}}\circ \pi_{\tilde a} \circ \jmath_{\tilde aa} 
 = \ad U_{p_{(o,\tilde a)}}\circ \ad U_{\tilde a a}\circ \pi_a  \\
& = \ad U_{p_{(o,\tilde a)}*(\tilde a a)}\circ \pi_a    
  = \ad U_{p_{(o,a)}}
  \circ \ad U_{p_{(a,o)}*p_{(o,\tilde a)}*(\tilde a a)}\circ \pi_a  = 
  \ad U_{p_{(o,a)}}\circ \pi_a   =  \si_a \ . 
\end{align*}
So, any quasi trivial representation defines in a natural way a Hilbert space representation. However 
these two representations are not, in general,  equivalent. In fact consider the unitary $T$ defined by $T_{a}:= U_{p_{(a,o)}}$, $a\in K$. Then $T_a\in(\si_a,\pi_a)$ for any $a\in K$, but 
$T$  does not intertwine the inclusion operators since 
$U_{\tilde aa}\, T_a =  U_{(\tilde a,a)* p_{(a,o)}} = T_{\tilde a} \, 
 U_{p_{o,\tilde a)}*(\tilde a,a)* p_{(a,o)}}$,
which is different from $T_a$ unless   the holonomy representation defined by $U$ 
is trivial.  
It is not surprising that to  any  Hilbert space representation $(\si,\mathbbm{1})$ of a net 
one can associate a quasi trivial representation carrying a nontrivial 
representation of the fundamental group. In fact, extend $\si$ to the representation $(\si^{\uparrow},\mathbbm{1})$
of the enveloping net bundle and  consider the covariant representation 
$({\si}^\uparrow_*,\mathbbm{1})$ of the holonomy dynamical system $(\overline{\cA}_*,\pi^o_1(K),\overline{j}_*)$. If the homotopy group 
has representations $V$ taking values in the commutant 
$({\si}^\uparrow_*,\si^{\uparrow}_*)$\footnote{If for instance $\si^\uparrow_*$ is irreducible (see Appendix \ref{AppRep}), then $\pi^o_1(K)$ must have 1-dimensional representations.},
then the pair $(\si^\uparrow_*,V)$ 
is still a covariant representation of the dynamical system. Turning back to the net,  
this yields a quasi trivial representation of the net not equivalent to $(\si,\mathbbm{1})$.
%
%

We now characterize those  nets whose enveloping net bundle is
trivial. 
\begin{proposition}
\label{Cb:10c}
The enveloping net bundle of a nondegenerate net is trivial if, and only if, the net 
has only quasi trivial representations.
\end{proposition}
\begin{proof}
$(\Rightarrow)$  Let $(\pi,U)$ be a representation of a net $(\cA,\jmath)_K$ with trivial
enveloping net bundle. 
Let $(\pi^\uparrow,U)$ be the extension of this representation to $(\overline{\cA},\overline{\jmath})_K$. Because of triviality $\overline{\jmath}_p =\mathrm{id}_{\overline{\cA}_o}$ for any $o\in K$ and for any loop $p$ over $o$. Therefore 
$\mathrm{ad}_{U_p}\circ \pi^\uparrow_o = \pi^\uparrow_o\circ \overline{\jmath}_p =
\pi^\uparrow_o$. By the definition of $\pi^\uparrow$,  
the quasi triviality of $(\pi,U)$ follows.\\
\indent $(\Leftarrow)$ Consider the holonomy dynamical system 
$(\overline{\cA}_*,\pi^o_1(K), \overline{\jmath}_*)$. Let $\rho$ be a faithful representation 
of the crossed product $\overline{\cA}_*\rtimes\pi^o_1(K)$. 
$\rho$ defines a covariant representation 
$(\rho,U^\rho)$  of the holonomy dynamical system and this, in turns, defines 
a representation  of the enveloping net bundle, Lemma \ref{Cb:8}, and so a representation 
of the net $(\cA,\jmath)_K$,  Lemma \ref{Cb:7}. 
Since only quasi trivial representations of the net are allowed, 
$U^\rho_{[p]}\in (\rho,\rho)$ for any  $[p]\in\pi^o_1(K)$. So  
$\rho\circ \overline{\jmath}_{*,[p]}= \mathrm{ad}U^\rho_{[p]}\circ\rho = 
\rho$ for any  $[p]\in\pi^o_1(K)$. 
Since $\rho$ is faithful $\overline{\jmath}_{*,[p]}=\mathrm{id}_{\overline{\cA}_*}$ for any  $[p]\in\pi^o_1(K)$, and the proof follows. 
\end{proof}

The next result relates representations and states. 
\begin{lemma}
\label{Cb:11}
For any $\rC^*$-net bundle $(\cA,\jmath)_K$, the following properties are equivalent.
\begin{itemize}
\item[(i)]   $(\cA,\jmath)_K$ has states.   
\item[(ii)]  There is a representation $(\pi,U)$ of $(\cA,\jmath)_K$ on a Hilbert space
             $\cH$ and a family $\Omega=\{\Omega_a \in \cH \}$ such that 
             $U_{oa}\Omega_a=\Omega_o$ for all $o\leq a$.
\item[(iii)] The holonomy dynamical system $(\cA_*,\pi_1^o(K),\jmath_*)$ admits a covariant 
             representation with an invariant vector.
\end{itemize}
\end{lemma}
\begin{proof}
$(i)\Rightarrow(ii)$ Let $\omega$ be a state of $(\cA,\jmath)_K$. 
Given  $o\in K$, let $(\pi_o,\cH_o,\Omega_o)$ be the 
GNS representation associated with the state $\omega_o$ of the algebra $\cA_o$. For any inclusion 
$o\leq a$ define 
\[
 U_{ao}\pi_o(A)\Omega_o:= \pi_a({\jmath}_{ao}(A))\Omega_a \ , \qquad A\in\cA_o \ . 
\]
A routine calculation shows that the $U_{ao}:\cH_o\to\cH_a$ are unitary operators with  
$U_{ao}\, \pi_o(\cdot )= \pi_a\circ {\jmath}_{ao}(\cdot) \,U_{ao}$. Moreover  
\[
U_{ea}\,U_{ao}\pi_o(\cdot)\Omega_o= U_{ea}\pi_a({\jmath}_{ao}(\cdot))\Omega_a = 
\pi_a({\jmath}_{ea}{\jmath}_{ao}(\cdot))\Omega_e = U_{eo}\pi_o(\cdot)\Omega_o \ . 
\]
Clearly, by  definition we have $U_{ao}\Omega_o=\Omega_a$.   
$(ii)\Rightarrow (iii)$ Define $\pi_* : \cA_* \to \cB(\cH)$ as in (\ref{Cb:9})
and consider $\Omega_o \in \cH$. Then the condition $U_{ao}\Omega_o = \Omega_a$,
$o \leq a$, implies that $U_p \Omega_o = \Omega_o$ for any loop $p$ over  $o$,
where $U_p$ is defined in (\ref{Cb:6}).
$(iii)\Rightarrow(i)$ 
Let $(\eta,V)$ be a covariant representation of 
$(\cA_*,\pi_1^o(K),\jmath_*)$ on the Hilbert space $\cH$ 
and $\zeta \in \cH$ be a $V$-invariant vector. 
We consider the representation $(\eta_*,V_*)$ of $(\cA_{**},\jmath_{**})_K$ 
constructed as in Lemma \ref{Cb:8} and define 
$\varphi_a(A):=(\zeta, \eta_{*,a}(A) \zeta)$, 
$A \in \cA_{**,a} = \cA_o$, 
$a \in K$. 
By $V$-invariance of $\zeta$, and covariance of $(\eta_*,V_*)$, we find
\[
\varphi_{\tilde a} \circ \jmath_{**, \tilde aa}(A) = 
(\zeta , ( \eta_{*,a} \circ \jmath_{**,\tilde aa}(A) ) \zeta) =
(V_{*,\tilde aa}\zeta , \eta_{*,a}(A) V_{*,\tilde aa}\zeta) =
\varphi_a(A)
\ .
\]
Thus $\varphi$ is a state of $(\cA_{**},\jmath_{**})_K$, and composing 
with the isomorphism with $(\cA,\jmath)_K$ yields the desired state.
\end{proof}

We conclude the section giving an easy consequence of the above results
in the setting of $G$-actions. 
Let $G$ be a symmetry group for $K$ and $(\cA,\jmath,\alpha)_K$ a $G$-covariant net. 
A \emph{$G$-covariant representation} of $(\cA,\jmath,\alpha)_K$ is a
representation $(\pi,U)$ of $(\cA,\jmath)_K$ such that the underlying 
family of Hilbert spaces $\cH := \{ \cH_o \}$ is endowed with unitary operators
$\Gamma^g_o : \cH_o \to \cH_{go}$, $g \in G$, $o \in K$, satisfying the relations 

\begin{equation}
\label{Cb:12}
\begin{array}{rclr}
\Gamma^h_{go}\, \Gamma^g_o & = & \Gamma^{hg}_o \ , & \ g,h\in G \ , \ \ o\in K  \ ,  \\[5pt]
\ad \Gamma^g_o \circ \pi_o & = & \pi_{go} \circ \alpha^g_o \ , &  \ g\in G \ , \ \ o\in K \ , \\[5pt]
\Gamma^g_{\tilde o} \, U_{\tilde oo} & =&  U_{\tilde go \ go} \, \Gamma^g_o \ , & \   
o\leq \tilde o \ , \ \ g\in G \ .
\end{array}
\end{equation}
Notice that when $(\pi,U)$ is topologically trivial,  $\Gamma$ induces a unitary 
representation of the symmetry group, since the inclusion operators $U_{ao}$ are constant.
Furthermore, recall that, when $G$ is a continuous symmetry group of the  poset,  
we assume, by convention, that the $G$-action on the net is continuous (see \S \ref{Ba}). 
Similarly, under  these circumstances we assume that 
for any $G$-covariant representation $(\pi,U,\Gamma)$, $\Gamma$  is \emph{strongly continuous}.
This amounts to saying that if $\{g_\lambda\}_{\Lambda}$ is a net in $G$ converging 
to the identity of the group then, for any  $o\in K$ and $a\in K$ with $a > o$, there exists  
$\lambda_a\in \Lambda$ such that $g_{\lambda}o\leq a$ for any $\lambda\in\Lambda$ with 
$\lambda\geq \lambda_a$ and  
\begin{equation}
\label{Cb:12a}
\| U_{a\, g_\lambda o}\, \Gamma^{g_\lambda}_o \Omega -\Omega\| \to 0 \ ,\qquad \forall \Omega\in\cH_a \ . 
\end{equation}
We have the following result. 
\begin{proposition}
\label{Cb:13}
Let $K$ be a poset with amenable fundamental group and 
$G$ an amenable symmetry group of $K$. 
Then every injective, $G$-covariant net of $\rC^*$-algebras over $K$ has a
$G$-covariant representation $(\pi, U,\Gamma)$. In particular, if 
$G$ is a continuous symmetry group of $K$, then $\Gamma$ is strongly continuous. 
\end{proposition}

\begin{proof}
Let $(\cA,\jmath,\alpha)_K$ be injective. Then by Rem.\ref{Bd:18} it suffices
to look for $G$-covariant representations of the enveloping net bundle
$(\overline{\cA},\overline{\jmath},\overline{\alpha})_K$. Now, by Prop.\ref{Ca:6}.ii   $(\overline{\cA},\overline{\jmath},\overline{\alpha})_K$ has a $G$-invariant state $\varphi$,
which, by Lemma \ref{Cb:11}, induces a GNS representation $\pi$.
Since $\varphi$ is $G$-invariant, setting
\[
\Gamma^g_o (\pi_o(T) \Omega_o) := ( \pi_{go} \circ \overline{\alpha}^g_o(T) ) \Omega_{go}
\ \ , \qquad  T \in \overline{\cA}_o \ , 
\]
for $o\in K$ and $g\in G$,  we easily find that $\pi$ is $G$-covariant.
Finally note that if $G$ is a continuous symmetry group 
of $K$, then  the action $\alpha$ is continuous by assumption (see \S \ref{Ba}).
As observed in Remark \ref{Bd:18}, the action $\overline{\alpha}$ of $G$ on the enveloping net bundle
is continuous as well. Using this, a routine calculation shows that $\Gamma$ 
is strongly continuous. 
\end{proof}


\section{Injectivity}
\label{D}
Injectivity  has been related to "outer properties" of the net: 
the existence of faithful embedding into $\rC^*$-net bundles or, in particular, 
the existence of faithful representations. In the present section we analyze how injectivity 
relates to 'inner properties' of the net. More precisely, our aim is to find   
conditions on the net itself, which  are either necessary or 
sufficient for injectivity. We shall discover that injectivity imposes 
a cohomological condition on the net. This will allow us 
to provide examples of degenerate nets, noninjective nets, and of injective nets
having no Hilbert space representation (for a further example of this last case see
Ex.\ref{ex_free}). A remarkable result is that 
any $\rC^*$-net bundle over $S^1$ whose \v Cech cocycle is globally defined is trivial.\bigskip


\

Our aim is to introduce a  simplicial set which will serve for defining the  \v Cech cocycle
mentioned above.  To begin with, let $K$ be a poset, and $S,F$ nonempty subsets of $K$. 
We shall write $F\leq S$ whenever any element of $F$ is smaller than any  elements
of $S$.  We now are ready to define the simplicial set $\Si^\circ_*(K)$:   for $n\geq 0$,  an \emph{$n$-simplex} 
is a string   $(F;o_{n+1},o_n,\ldots, o_1)$ where $o_{n+1},o_n,\ldots, o_1$  
are  element of $K$, called the \emph{vertices} of the $n$-simplex, and $F$ is a nonempty subset of $K$, called the \emph{support} of the $n$-simplex, 
satisfying  the relation  $F\leq \{o_{n+1},o_n,\ldots, o_1\}$. This simplicial set is symmetric 
since the string whose vertices are a permutation of the vertices of an $n$-simplex 
$(F;o_{n+1},o_n,\ldots, o_1)$ and whose support is  $F$ is an $n$-simplex as well. As usual 
$\Si^\circ_n(K)$ will denote the set of $n$-simplices. \\ 
\indent We denote  the set 
of  symmetric subsimplicial sets $D_*$ of  $\Si^\circ_*(K)$, with $D_i\subseteq \Si^\circ_i(K)$
for any $i$, by $\mathrm{Sub}(\Sigma^\circ_*(K))$.  
Note that $\mathrm{Sub}(\Sigma^\circ_*(K))$ is closed under 
finite or infinite union, and  finite or infinite intersection, if not empty.   
Elements of $\mathrm{Sub}(\Sigma^\circ_*(K))$
are, for instance,  $\Si^\circ_*(S)$  for any nonempty subset  $S\subseteq K$. 
%
%
For any $0$-simplex $(F,o)$ of $D_0$, we define the $\rC^*$-algebra
\[
\cA^F_o := \rC^* \{ \jmath_{oa}(\cA_a) \ | \ a \in F \}\subseteq \cA_o \ .
\]
\begin{definition}
\label{Da:0}
Let $(\cA,\jmath)_K$ be a net of $\rC^*$-algebras and $D_*\in \mathrm{Sub}(\Sigma^\circ_*(K))$.  
A \textbf{ (generalized) \v Cech cocycle} of $(\cA,\jmath)_K$  defined over  $D_*$  is a family 
$\zeta:=\{ \zeta^F_{\tilde oo} \ | \ (F;\tilde o,o)\in D_1\}$ of  $^*$-isomorphisms 
\[ 
\zeta^F_{\tilde oo} : \cA^F_o \to \cA^F_{\tilde o} \ , 
\qquad \forall   (F;o,\tilde o)\in D_1  \ , 
\]
satisfying, for any $1$-simplex $(F;\tilde o,o)$,  the relation
\begin{equation}
\label{Da:1}
\zeta^F_{\tilde oo} \circ \jmath_{oa} = \jmath_{\tilde oa}, \qquad \forall a\in F \ .
\end{equation}
%
%
\end{definition}
Before showing the relation to injectivity, it is 
convenient to draw  some  interesting  consequences of the above definition.  \emph{First},  
the defining relation implies that   $\zeta$ satisfies 
\begin{equation}
\label{Da:2}
\zeta^F_{o_3o_2}\circ \zeta^F_{o_2o_1} = \zeta^F_{o_3o_1} \ , \qquad   (F;o_1,o_2,o_3)\in D_2 \ , 
\end{equation}
and 
\begin{equation}
\label{Da:2a}
\zeta^F_{o_2o_1} = {\zeta^F_{o_1o_2}}^{-1} \ , \qquad   (F;o_1,o_2)\in D_1 \ . 
\end{equation}
The first relation, the \emph{cocycle equation}\footnote{This cohomology generalizes the \v Cech cohomology for  the dual poset introduced  in \cite{RRV}.}, follows observing 
that  
\[
\zeta^F_{o_3o_2}\circ \zeta^F_{o_2o_1}\circ \jmath_{o_1o} = 
\zeta^F_{o_3o_2}\circ \jmath_{o_2o} = \jmath_{o_3o} = \zeta^F_{o_3o_1}\circ \jmath_{o_1o} \ , 
\]
for any $o\in F$, where we have applied (\ref{Da:1}). A similar reasoning leads to 
the second relation. \emph{Secondly}, the \v Cech cocycle reduces 
to the inclusion maps when  defined on 1-simplices $(F;a,o)$  such that $o\leq a$, that is 
\[
\zeta^F_{ao}=\jmath_{ao}\restriction \cA^F_o \ , 
\]
as an easy consequence of equation (\ref{Da:1}).
\begin{lemma}
\label{Da:4}
Any net $(\cA,\jmath)_K$ has a unique \v Cech cocycle.  
\end{lemma}
\begin{proof}
\emph{Existence}. Take a pair $o_1,o_2\in K$ having a common minorant $o$. Consider 
$\Si^\circ_*(\{o_1,o_2,o\})$. Any 1-simplex of $\Si^\circ_1(\{o_1,o_2,o\})$  has the form
$(o;\tilde a, a)$ with $a,\tilde a\in\{o_1,o_2\}$. Then, it is easily seen  that the collection 
$\zeta:=\{\zeta^o_{\tilde a a} \, , \, a,\tilde a\in\{o_1,o_2\}\}$, where 
\[
\zeta^{o}_{\tilde a a} :
\jmath_{a o}(\cA_o) \to \jmath_{\tilde a o}(\cA_o)
\ \ , \ \
\zeta^{o}_{\tilde a a} \circ \jmath_{ao}(A) := \jmath_{\tilde ao}(A)
\ , \
\forall A\in\cA_o \ , 
\ 
\]
is a \v Cech cocycle defined over $\Si^\circ_*(\{o_1,o_2,o\})$.  \\
\indent \emph{Uniqueness}. Let $\zeta^\alpha$, $\alpha\in\Lambda$, 
be the collection of all \v Cech
cocycles of the net defined, respectively, over $D^\alpha_*$, $\alpha\in\Lambda$. Set
\[
\zeta^F_{\tilde oo}:= \zeta^{\alpha,F}_{\tilde oo} \  \qquad  if \qquad  (F;\tilde o,o)\in D^\alpha_1 \ .  
\]
The definition is well posed: if $(F;\tilde o,o)\in D^\alpha_1\cap D^\beta_1$  for $\alpha,\beta\in\Lambda$,   then by (\ref{Da:1}), 
$\zeta^{\alpha,F}_{\tilde oo}\circ \jmath_{oa} = \jmath_{\tilde oa} = 
\zeta^{\beta,F}_{\tilde oo} \circ \jmath_{oa}$, for any $a\in F$.
So $\zeta^{\alpha,F}_{\tilde oo}=\zeta^{\beta,F}_{\tilde oo}$, because they coincide 
on the generators of the algebra $\cA^F_{o}$. Therefore $\zeta$ is a \v Cech cocycle defined 
over $\cup_{\alpha \in \Lambda} D^\alpha_*$. 
\end{proof}
We have established the  existence and the uniqueness of the \v Cech cocycle of a net.
Moreover, it is clear by the above proof that 
$\cup_{\alpha \in \Lambda} D^\alpha_*$  is the largest  element 
of $\mathrm{Sub}(\Si^\circ_*(K))$ where the \v Cech cocycle is defined. 
\begin{definition} 
Given a net $(\cA,\jmath)_K$. 
We refer to the largest element of $\mathrm{Sub}(\Si^\circ_*(K))$ where 
the \v Cech cocycle of the net is defined as the \textbf{the domain} of the cocycle; 
we shall say that the cocycle is \textbf{globally defined} if the domain equals $\Si^\circ_*(K)$. 
\end{definition}
From now on our purpose will be to characterize the domain of the \v Cech cocycle
of the net and to establish the relation to injectivity. 
Part of the information about the domain of the \v Cech cocycle is transmitted 
under a  faithful embedding. This is the content of the next lemma. 
\begin{lemma}
\label{Da:5}
Let $\psi:(\cA,\jmath)_K\to (\cB,\imath)_K$ be a monomorphism.  
Then the domain $D_*$ of the \v Cech cocycle $\xi$ of $(\cB,\imath)_K$ is contained 
in the domain the \v Cech cocycle of $(\cA,\jmath)_K$. 
\end{lemma}
\begin{proof}
Given $(F;\tilde o,o)\in D_1$ we have
\[
\xi^F_{\tilde oo}\circ \psi_o \circ \jmath_{oa} = 
\xi^F_{\tilde oo}\circ \imath_{oa} \circ \psi_a = \imath_{\tilde o a} \circ \psi_a   = 
\psi_{\tilde o}\circ \jmath_{\tilde o a}
\ , 
\] 
for any $a\in F$,  where we have used equation (\ref{Da:1}). So 
$(\xi^F_{\tilde oo}\circ \psi_o)(\cA^F_o) = \psi_{\tilde o}(\cA^F_{\tilde o})$. This  
and the uniqueness of the \v Cech cocycle, implies that the composition  
\[
{\psi}^{-1}_{\tilde o}\circ \xi^F_{\tilde oo}\circ \psi_o \ , \qquad (F;\tilde o,o)\in D_1 \ ,  
\]
is the \v Cech cocycle of the net $(\cA,\jmath)_K$, and the proof follows. 
\end{proof}
We now make a first  step toward the understanding of the r\^ole played by 
the \v Cech cocycle in the theory.  
\begin{lemma}
\label{Da:6}
Let $(\cA,\jmath)_K$ be a $\rC^*$-net bundle and let $D_*$ be the domain of the  \v Cech cocycle. Then given $(F;\tilde o,o)\in D_1$
we have
\begin{equation}
\label{Da:7}
\zeta^F_{\tilde oo}= \jmath_{\tilde oa}\circ {\jmath}_{ao} \ , \qquad \forall a\in F \ , 
\end{equation}
and  
\begin{equation}
\label{Da:8}
{\jmath}_{a_2o}\circ {\jmath}_{oa_1} =  {\jmath}_{a_2\tilde o}\circ {\jmath}_{\tilde oa_1} \, \qquad \forall a_1,a_2\in F \ .  
\end{equation}
\end{lemma}
\begin{proof}
First of all note that $\cA^F_o=\cA_o$, for any $F\leq o$, since we are considering a $\rC^*$-net bundle. Then the first relation derives directly from (\ref{Da:1}). Using the first relation  
we have 
$\jmath_{\tilde oa_1}\circ {\jmath}_{a_1o} = \jmath_{\tilde oa_2}\circ {\jmath}_{a_2o}$, 
for any $a_1,a_2\in F$, 
and the second relation follows.  
\end{proof}
The second equation looks like a triviality result for the $\rC^*$-net bundle. 
In particular, when the \v Cech cocycle is globally defined, this results 
asserts that the 1-cocycle defined by $\jmath$,  
giving the action of the homotopy group of the holonomy dynamical system (see (\ref{Bc:3})), 
does not depends on the support of the 1-simplex: $\jmath_b = \jmath_{\tilde b}$ 
for any pair of 1-simplices such that $\partial_ib=\partial_i\tilde b$ for $i=0,1$. 
We shall see soon that any $\rC^*$-net bundle over $S^1$ having a globally defined \v Cech 
cocycle is indeed trivial. \\
\indent We now start analyzing the relation between injectivity and the \v Cech cocycle. 
\begin{lemma}
\label{Da:9}
Let $(\cA,\jmath)_K$ be a net of $\rC^*$-algebras and
$P \subseteq K$ pathwise connected and such that $(\cA,\jmath)_P$ has a faithful
Hilbert space representation. Then $\Si^\circ_*(P)$ is contained in the domain of the 
\v Cech cocycle of the net $(\cA,\jmath)_K$.  
\end{lemma}

\begin{proof}
By hypothesis there exists a $\rC^*$-algebra $\cA(P)$ (the 
universal one) and unital faithful ${^*}$-morphisms 
$\psi_o : {\cA_o} \to \cA(P)$, with $o\in P$, 
such that  $\psi_{o'} \circ \jmath_{o'o}= \psi_o$ for any $o\leq o'$.
Now, given $(F;o_2,o_1)\in \Si^\circ_1(P)$, for any $a_1,\ldots, a_n\in F$, define 
\begin{equation}
\label{Da:10}
 \zeta^F_{o_2o_1}\big(\jmath_{o_1a_1}(A_1) \cdots  \jmath_{o_1a_n}(A_{n})\big) 
 :=
 \jmath_{o_2a_1}(A_1) \cdots \jmath_{o_2a_n}(A_n) 
\end{equation}
where $A_i\in\cA_{a_i}$ for $i=1,\ldots,n$, and extend $\zeta^F_{o_2o_1}$ by 
linearity to all the ${^*}$-algebra generated by $\jmath_{{o_1a}}(A_{a})$ as  
$a$ varies in $F$. To prove that $\zeta^F_{o_2o_1}$ is isometric we make use 
of the $\psi_o$'s as follows:
\begin{align*}
 \|\jmath_{o_2a_1}(A_1)&  \jmath_{o_2a_2}(A_2) \cdots   \jmath_{o_2a_n}(A_n)\|  = \\
& =  \norm{\psi_{o_2} \big(\jmath_{o_2a_1}(A_1)\jmath_{o_2a_2}(A_2)\cdots  \jmath_{o_2a_n}(A_n)\big)}  \\
& = \norm{(\psi_{o_2}\circ\jmath_{o_2a_1})(A_1)\,  (\psi_{o_2}\circ\jmath_{o_2a_2})(A_2)\cdots  (\psi_{o_2}\circ\jmath_{o_2a_n})(A_n)} \\
& = \norm{\psi_{a_1}(A_1) \psi_{a_2}(A_2)\cdots  \psi_{a_n}(A_n)} \\
& = \norm{(\psi_{o_1}\circ\jmath_{o_1a_1})(A_1)\, (\psi_{o_1}\circ\jmath_{o_1a_2})(A_2)\cdots  (\psi_{o_1}\circ\jmath_{o_1a_n})(A_n)} \\
& = \norm{\psi_{o_1} \big(\jmath_{o_1a_1}(A_1)\jmath_{o_1a_2}(A_2)\cdots  \jmath_{o_1a_n}(A_n)\big)}  \\
& = \norm{\jmath_{o_1a_1}(A_1)\jmath_{o_1a_2}(A_2)\cdots  \jmath_{o_1a_n}(A_n)}  \ .
\end{align*}
The same reasoning applies to finite linear combinations, hence $\zeta^F_{o_2o_1}$ extends by 
continuity to a ${^*}$-isomorphism from $\cA^F_{o_1}$ to $\cA^F_{o_2}$ that, by definition,
fulfills (\ref{Da:1}). 
\end{proof}
The following consequence of Lemma \ref{Da:9} will be useful to determine whether 
a given net is injective (see Example \ref{Da:12}).
\begin{corollary}
\label{Da:11}
 Let $(\cA,\jmath)_K$ be an injective net. Then $\Si^\circ_*(P)$ 
 is  contained in the domain of the \v Cech cocycle
 for any connected and simply connected  subset $P$ of $K$.
\end{corollary}
Two observations are in order. \emph{First}, the above result applies to 
$\rC^*$-net bundles since any $\rC^*$-net bundle is injective by definition. 
\emph{Secondly}, even if  any poset $K$ is covered by the set of its 
connected and simply connected subsets,  
the \v Cech cocycle of an injective net may not be  globally defined since 
$\Si^\circ_*(S)\cup \Si^\circ_*(P)$ may be  smaller than $\Si^\circ_*(S\cup P)$,  
for any pair $S,P$ of nonempty subsets of $K$. 
In fact, if  there are $o\in S\setminus P$, $a\in P\setminus S$ and a nonempty 
subset $F$ of $S\cup P$ with $F\leq \{o,a\}$, then the 1-simplex $(F;a,o)\in \Si^\circ_1(S\cup P)$ but 
$(F;a,o)\not\in \Si^\circ_1(S)\cup \Si^\circ_1(P)$. 
%


\subsection{Examples}
\label{Daa}

Using the results of the previous section we give examples of degenerate nets. 
Afterwards, we make some conjecturs concerning the domain of the \v Cech cocycle 
and the injectivity of the nets. At the end of the section 
we give interesting examples of nets coming from complexes of groups.\smallskip

To begin with we give an example of a degenerate net.  
\begin{example}
\label{Da:12}
We consider the poset $B$ with elements $\{m,o,a,x,y\}$ and order relation 
$m\leq o,a\leq x,y$, so there is a minimum  $m$ and two maximal elements $x,y$. This poset 
is simply connected since it is downward directed (see preliminaries). \\
\indent We now define a net over $B$. Let $\cC$ be a unital $\rC^*$-algebra. 
The fibres are defined as follows: $\cC_m:= \cC$ and $\cC_a= \cC_o= \cC_x= \cC_y =\bM_2(\cC)$.  
Passing to the inclusion maps, let 
$\rho:\cC\to \bM_2(\cC)$ be the ampliation of $\cC$, that is the faithful ${^*}$-morphism
defined by
\[
\rho(A):=
\left(
\begin{array}{cc}
A & 0 \\
0 & A
\end{array}
\right) \ , \qquad  A\in \cC \ . 
\]
Let $\gamma:\bM_2(\cC)\to \bM_2(\cC)$ be the automorphism 
defined by $\gamma:=\mathrm{Ad}_V$ where  
$V$ is the unitary of $\bM_2(\cC)$ defined by 
\[
 V :=
\left(
\begin{array}{cc}
0 & 1 \\
1 & 0
\end{array}
\right) \ . 
\]
Define $\jmath_{xm}=\jmath_{ym}=\jmath_{am}=\jmath_{om}= \rho$,  
$\jmath_{xa}=\jmath_{ya}=\mathrm{id}_{\bM_2(\cC)}$,  
$\jmath_{xo}=\mathrm{id}_{\bM_2(\cC)}$ and $\jmath_{yo}=\gamma$.
Since  $\gamma\circ \rho =\rho$, it is easily seen that this system is 
a net of $\rC^*$-algebras denoted by $(\cC,\jmath)_{B}$. \\
\indent We now prove that 
\emph{the net $(\cC,\jmath)_{B}$ is not injective}.
Assuming that $(\cC,\jmath)_{B}$ is injective, and $B$
being simply connected, the \v Cech cocycle should be globally defined by 
Cor.\ref{Da:11}. 
Consider the projection $E \in \bM_2(\cC)$ defined by 
\[
 E :=
\left(
\begin{array}{cc}
1 & 0 \\
0 & 0
\end{array}
\right) \ . 
\] 
Since $\gamma(E)=1-E$ we have $\jmath_{ya}(E)\,\jmath_{yo}(E) = E\,\gamma(E) = 0$, while   
$\jmath_{xa}(E)\, \jmath_{xo}(E) = E$. Therefore, fixing $F := \{ a,o \}$,  
an isomorphism $\zeta^F_{xy} : \cA^F_y \to \cA^F_x$ fulfilling
$\zeta^F_{xy} \circ \jmath_{yo} = \jmath_{xo}$
cannot exist, since  we should have 
\[
0 = \zeta^F_{xy} \ (\jmath_{ya}(E)\,\jmath_{yo}(E)) = 
\zeta^F_{xy}(\jmath_{ya}(E))\,\zeta^F_{xy}(\jmath_{yo}(E)) = 
\jmath_{xa}(E)\,\jmath_{xo}(E)=E \ ,
\]
a contradiction. Finally note 
that when $\cC$ is a \emph{simple} $\rC^*$-algebra, the 
net $(\cC,\jmath)_B$ is degenerate. This is because the fibres of $(\cC,\jmath)_B$ are 
simple, so any representation
is trivial.
As a consequence, the enveloping net bundle of $(\cC,\jmath)_B$ is the null net bundle.
\end{example}
The next is an example of an injective net having no Hilbert space 
representations.
\begin{example}
\label{Da:12a}
The poset we consider is obtained from the one
of the previous example by removing the minimum, so it is given by 
$C_2 := \{a,o,x,y\}$ with order relation $a,o\leq x,y$. 
We call $C_2$ a {\em 2-cylinder}, a terminology that will be clarified in \cite{RVs1}.
One can easily see that this poset is not simply connected and that the homotopy group 
is $\bZ$.
We now define a net 
over $C_2$. Using the same notation as in the  previous example, let $\cC$ be 
a simple $\rC^*$-algebra, and  set 
$\cC_{o}=\cC_{a}=\cC_{y}=\cC_{x}= \mathbb{M}_2(\cC)$, and  
$\jmath_{xa}=\jmath_{ya}=\mathrm{id}_{\mathbb{M}_2(\cC)}$,  
$\jmath_{xo}=\mathrm{id}_{\mathbb{M}_2(\cC)}$, 
$\jmath_{yo}=\gamma$.  Since there are no compositions of inclusions in this poset, 
$(\cC,\jmath)_{C_2}$ is a $\rC^*$-net bundle (all the inclusions are isomorphisms). 
$(\cC,\jmath)_{C_2}$ is injective by definition, but the \v Cech cocycle of this net 
is not globally defined for  the same reason as in the previous example.
For, take the projection $E$ and the subset $F=\{o,a\}$ of the previous example.
Then the \v Cech cocycle is not defined on 1-simplex  $(F;x,y)$, since 
as above we should have $0 = \zeta^F_{xy} \ (\jmath_{ya}(E)\,\jmath_{yo}(E)) = E$.
Hence, by Lemma \ref{Da:9},  the $\rC^*$-net bundle $(\cA,\jmath)_{C_2}$ has  no Hilbert 
space representations because  the fibres are simple $\rC^*$-algebras.
This also implies that the universal $\rC^*$-algebra of $(\cA,\jmath)_{C_2}$ 
is trivial, i.e. $\cA^u = \{ 0 \}$.
\end{example}
We now analyse  the relation between injectivity and 
the domain of the \v Cech cocycle in the opposite sense, and ask whether 
a net is injective on subsets of the poset associated with the domain of the \v Cech cocycle.  
We have no general result, however it is our opinion that the following 
assertions should hold. \emph{First}, given a nondegenerate net. If the \v Cech cocycle 
is defined on $\Si^\circ_*(P)$ for any connected and simply connected subset $P$ of the poset, 
then the net is should be injcetive; If the \v Cech cocycle is globally defined, then 
the net should have faithful Hilbert space representations. \emph{Secondly}, a 
$\rC^*$-net bundle is trivial if its \v Cech cocycle is globally defined. \\
\indent The next is an example supporting the first assertion. 
\begin{example}
\label{Da:13}
We consider the 2-cylinder $C_2 := \{ o,a,x,y \}$ of the previous example. 
Let $(\cA,\jmath)_{C_2}$ be a net of $\rC^*$-algebras. 
Assume that this net has a globally defined \v Cech cocycle. 
Then in particular we have a $^*$-isomorphism $\zeta^{\{o,a\}}_{yx} : \cA^{\{o,a\}}_x \to \cA^{\{o,a\}}_y$
such that 
\begin{equation}
\label{Da:1a}
\zeta^{\{o,a\}}_{yx} \circ \jmath_{xe} = \jmath_{ye}
\ , \
\zeta^{\{o,a\}}_{xy} \circ \jmath_{ye} = \jmath_{xe}
\ \ , \ \
e=o,a
\ .
\end{equation}
Denoting a copy of $\cA^{\{o,a\}}_x$ by $\cC$ yields the diagram
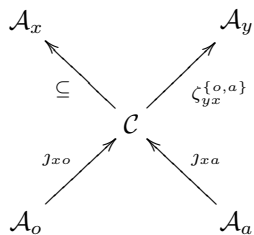
\[
\xymatrix{
   \cA_x
&   
&  \cA_y
\\ 
&  \cC
   \ar[ru]_{\zeta^{\{o,a\}}_{yx}}
   \ar[lu]^{\subseteq}
&
\\
\cA_o
\ar[ru]^{\jmath_{xo}}
&
&
\cA_a
\ar[lu]_{\jmath_{xa}}
}
\]
By (\ref{Da:1a}), the desired Hilbert space representation is obtained 
as a representation of the amalgamated free product $\cA_x *_\cC \cA_y$.
\end{example}
A proof of the first conjecture seems to be far from to be reached. 
The above example suggests that it may be related to the "realizability" of 
the generalized amalgamated free product of $\rC^*$-algebras, \cite{Bil}.\\
\indent The next result supports the second  assertion. 
\begin{lemma}
\label{Da:14}
Any $\rC^*$-net bundle over $S^1$ having a globally defined \v Cech cocycle is trivial. 
\end{lemma}
\begin{proof}
Let $(\cA,\jmath)_\cI$ be a $\rC^*$-net bundle over $S^1$ (i.e., $\cI$ is the set of open 
connected intervals of $S^1$ whose closure is properly contained in $S^1$).  
The homotopy group of $\cI$ is the same as that of $S^1$, i.e. $\bZ$. 
We now consider a path in $\cI$, defined as follows. Let $x,y\in \cI$ be such that 
$x\cup y= S^1$ and $x\cap y$ has two connected components $a$ and $o$ respectively. 
Clearly $a,o\in\cI$.  Consider the 1-simplices $b_1$ and $b_2$ defined by 
\[
|b_1|:=x \ ,  \ \partial_1b_1:=a \ , \   \partial_0b_1:= o \ \ , \ \   
|b_2|:=y \ ,  \partial_1b_2:=o \ , \   \partial_0b_2:= a \ .   
\]
The path $b_2*b_1$ is a loop over $a$ which is not homotopically trivial and whose  
image under the isomorphism $\pi_1(\cI) \simeq \pi_1(S^1)$ is the generator
of $\pi_1(S^1)$. 
As a consequence, $(\cA,\jmath)_\cI$ is trivial if the
isomorphism $\jmath_{b_2*b_1}$ generating the holonomy of
$(\cA,\jmath)_\cI$ is the identity.
Now, since the \v Cech cocycle is globally defined,  it is defined, in particular, on the 1-simplex  
$(\{x,y\};a,o)$. By using equation (\ref{Da:8}) (within Lemma \ref{Da:6}) 
we have that 
${\jmath}_{ox}\circ \jmath_{xa}  = {\jmath}_{oy}\circ \jmath_{ya}$,  
so 
$\jmath_{b_2*b_1}  = 
{\jmath}_{ay}\circ \jmath_{yo}\circ {\jmath}_{ox}\circ \jmath_{xa}  = 
\mathrm{id}_{\cA_a}$,  
where the homotopy equivalence of the paths  $b_2*b_1$ and 
$\overline{(ya)} *(yo)* \overline{(xo)}*(xa)$ has been used. 
\end{proof}
Note that if a net has a globally defined \v Cech cocycle, then the  
\v Cech cocycle of the enveloping net bundle may be not globally defined. 
In algebraic quantum field theory there are examples of nets  
having both faithful Hilbert space representations and faithful representations 
carrying a nontrivial representation of the homotopy group of the poset which are not 
quasi trivial (see \cite{BR, BFM}). 
So the enveloping net bundle is not trivial, Prop.\ref{Cb:10c}.\smallskip

In conclusion we present interesting examples of injective and degenerate nets coming from 
geometric group theory. 
\begin{example} 
\label{Da:15}
The \emph{triangle} poset is a set $T$ of seven elements $\{m,A,B,C,X,Y,Z\}$ 
ordered as follows: $m\leq A,B,C$ and $ A,B\leq X$, $B,C\leq Y$ and $A,C\leq Z$. 
$T$ is downward directed, so it is simply connected. Any net of groups 
$(G,y)_T$ over $T$ is a triangle of groups and,  
by the functor $\rC^*$ (\S \ref{Bbc}), a net of group $\rC^*$-algebras 
$(\rC^*(G),\rC^*(y))_K$ is associated with $(G,y)_T$. By the properties of functor $\rC^*$,
the net $(G,y)_T$ embeds faithfully in the colimit group if, and only if, 
the net $(\rC^*(G),\rC^*(y))_K$ embeds faithfully in the universal $\rC^*$-algebra $\rC^*(G)^\ru$. 
Gersten and Stallings (\cite{Sta}) gave sufficient conditions for the existence of this faithful embedding in  terms 
of angles associated to a triangle of groups. 
However, 
they also gave an example of a triangle of groups which does not satisfy these conditions and 
has a trivial colimit: take 
$H_m:=\mathbbm{1}$ the trivial group,
$H_A:=\{ a \}$, $H_B:=\{ b \}$ and $H_C:=\{ c \}$ the free groups of one generator,
$H_X:= \{ a,b \, | \, b^2=aba^{-1} \}$, $H_Y:=\{ b,c \,|\, c^2=bcb^{-1} \}$ and 
$H_Z:=\{a,c\,|\, a^2=cac^{-1}\}$. 
Taking the inclusion of groups as inclusion maps 
we get a net of groups $(H,i)_K$. The colimit of this net is the Mennicke
group $M(2,2,2)=\{ a,b,c\,|\, b^2=aba^{-1} \, , \,  c^2=bcb^{-1}\, ,\,  a^2=cac^{-1}\}$ 
which is known to be trivial.  So, the corresponding net of group $\rC^*$-algebras 
is degenerate. 
\end{example}
Two observations about this example are in order. \emph{First}, note that  
$(\rC^*(H),\rC^*(i))_{T}$ is  a degenerate net of $\rC^*$-algebras having a globally defined \v Cech cocycle (this can be easily seen). 
\emph{Secondly}, it is interesting to analyze the r\^ole of the minimum $m$. Consider the poset 
$T\setminus\{m\}$ obtained by removing $m$ from $T$. Any net of $\rC^*$-algebras over $T\setminus\{m\}$ is injective since any  such a net $(\cA,\jmath)_{T\setminus\{m\}}$ admits faithful representations. To prove this we use an idea due to Blackadar (\cite{Bla}). Let $\kappa$ be 
a cardinal greater than the cardinality of the algebra $\cA_a$ for any element $a$ of  $T\setminus\{m\}$. Then we set $\pi_a$ as the tensor product of the universal representation  of 
$\cA_a$ times  $1_{\kappa}$. The trick of the cardinality implies that 
$\pi_a\circ \jmath_{ao}$ is unitarily 
equivalent to $\pi_o$ for any inclusion $a\leq o$. So, choose such a unitary opertor $V_{ao}$ for any inclusion $o\leq a$. 
Then, since in $T\setminus\{m\}$ conpositions of inclusions are not possible the pair 
$(\pi,V)$ is a faithful representation of $(\cA,\jmath)_{T\setminus\{m\}}$. Applying this result 
to the example of Gersten and Stallings, we have that  the net  
$(\rC^*(H),\rC^*(i))_{T\setminus\{m\}}$ is injective and that 
the net $(H,i)_{T\setminus\{m\}}$ has faithful representations. However, 
$(\rC^*(H),\rC^*(i))_{T\setminus\{m\}}$ has no Hilbert space representations, because 
any Hilbert space representation of this net can be easily extended to a Hilbert space representation of $(\rC^*(H),\rC^*(y))_{T}$ and this, as observed before, is not possible.


\section{Comments and outlook}
\label{Z}

We list some topics and questions arising from the present
paper. \smallskip

\noindent 1.  The examples of nets of $\rC^*$-algebras defined by groups of loops, 
              \S  \ref{Bb}, might have  interesting applications in algebraic 
              quantum field theory. They are model independent and are constructed using only the 
              principles of the theory. These examples  are 
              still incomplete (see Remark \ref{Bbc:2}); a complete description 
              will be given in a forthcoming paper \cite{CRV}. \\[5pt]
\noindent 2.  
              %
              The present paper poses some new questions, in particular whether
              an injective net with a globally defined \v Cech cocycle 
              has faithful Hilbert space representations.
              The existence of  Hilbert space representations is equivalent 
              to proving that a suitable ideal of the fibre of the enveloping net bundle is proper 
              (Prop. \ref{Cb:10ab}). So, one should find a relation  
              between the \v Cech cocycle of the net to this ideal. Examples 
              where this ideal is not proper have been given in \S \ref{Daa}.\\[5pt]  
\noindent 3. As an application of the results of the present work 
             we shall show, in a forthcoming paper \cite{RVs1}, the injectivity of nets 
             over $S^1$, an interesting class due to the
             applications in conformal field theory.\\[5pt]
 \noindent 4. Finally, we stress that nets of $\rC^*$-algebras defined over a (base for the topology
              of a) space $X$ carry topological information even when they are not net bundles. 
              In fact, it can be proved that every net is a precosheaf of local sections 
              of a canonical $\rC^*$-bundle  (in the topological sense). 
              This result, and its consequences, is the object of a forthcoming paper (\cite{RV2}).


\appendix

\section{Elementary properties of representations}
\label{AppRep}

In this appendix we give some basic properties of representations of nets of 
$\rC^*$-algebras. We will give particular emphasis to the fact that these can be
conveniently described in terms of the covariant representations of the
dynamical system associated with the enveloping net bundle of the given net.

More precisely, let $(\cA,\jmath)_K$ be a net of $\rC^*$-algebras and 
$( \overline \cA , \overline \jmath )_K$ denote the enveloping net bundle
(see \S \ref{Bd}); then by Prop.\ref{Bc:10} we have a dynamical
system
$( \overline \cA_* , \pi_1^o(K) , \overline \jmath_* )$,
where $\overline{\cA}_* := \overline{\cA}_o$ for some fixed $o \in K$.
A  representation  $(\pi,U)_K$ of   $(\cA,\jmath)_K$  extends, by 
Lemma \ref{Cb:7},  to a representation
$( {\pi}^{\uparrow} , U )$
of
$( \overline \cA , \overline \jmath )_K$
which, in turns, yields the $\pi_1^o(K)$-covariant representation
\begin{equation}
\label{eq.Arep}
{\pi}^{\uparrow}_* : \overline{\cA}_*  \to \cB(\cH_o)
\ \ , \ \
U_* : \pi^o_1(K) \to \cU(\cH)
\end{equation}
(see Lemma \ref{Cb:8}). We shall keep the above notation 
throughout the appendix.

\subsection{Decomposition of representations}

Let $(\cH,U)_K$ denote a Hilbert net bundle. 
Then the $\rC^*$-net bundle $(\cB \cH , \ad U)_K$ is defined,
and we say that a family 
$T = \{ T_a \in \cB(\cH_a) \}$ 
is a {\em section} of $(\cB \cH,\ad U)_K$ whenever the following relations
are fulfilled,
\[
\ad U_{\tilde aa}(T_a) = T_{\tilde a} 
\ \ , \ \ 
\forall a \leq \tilde a \ .
\]
We denote the set of sections of $(\cB \cH , \ad U)_K$ by $\Si(\cB \cH , \ad U)_K$; 
this is a unital $\rC^*$-algebra under the natural $*$-algebraic structure
\[
T+T' := \{ T_a + T'_a \}
\ \ , \ \
T^* := \{ T_a^* \}
\ \ , \ \
TT' := \{ T_a T'_a \}
\ ,
\]
and $\rC^*$-norm defined by $\| T \| := \| T_o \|$, where $o \in K$ is fixed.
(Since  $K$ is pathwise connected, for every $a \in K$ there is
a path $p : a \to o$, so that $\| T_o \| = \| \ad U_{*,[p]} (T_a) \| = \| T_a \|$.)

\begin{lemma}
\label{lem.rep01}
$\Si(\cB \cH , \ad U)_K$ is isomorphic to the $\rC^*$-algebra
of operators in $\cB(\cH_o)$ that are invariant under the holonomy representation
\begin{equation}
\label{eq.holrep}
\ad U_* : \pi_1^o(K) \to {\mathrm{Aut}} \cB(\cH_o) \ .
\end{equation}
\end{lemma}

\begin{proof}
If $T \in \Si(\cB \cH , \ad U)_K$, then it easily follows fromthe definition of 
the holonomy representation (\ref{eq.u*}) that $T_o$ is $\ad U_*$-invariant.
%
On the converse, 
fix a path frame $P_o=\{ p_{(a,o)}, \ a\in K \}$. If $T_o\in \cB(\cH_o)$ is $\ad U_*$-invariant, 
define  
\[
 T_a := \ad U_{p_{(a,o)}} (T_o) \ , \qquad a\in K \ .
\] 
For any inclusion $a\leq \tilde a$ we have that 
$\ad U_{\tilde aa} (T_a) =$  
$\ad U_{p_{(\tilde a,o)}}\circ \ad \{U_{p_{(o,\tilde a)}}\,U_{\tilde ao}\,U_{p_{(a,o)}}\} (T_o)$ 
$=\ad U_{p_{(\tilde a,o)}} (T_o) =  T_{\tilde a}$, by $\ad U_*$-invariance.
thus $T$ is a section and the Lemma is proved.
\end{proof}

Now, let  $(\pi,U)$ be a representation of $(\cA,\jmath)_K$ over the family
of Hilbert spaces $\cH := \{ \cH_a \}$, so that $(\cH,U)_K$ 
is a Hilbert net bundle. We define
\begin{equation}
\label{d.piU}
\Si (\pi,U) 
:= 
\{ T \in \Si(\cB \cH , \ad U)_K : 
   T_a \in (\pi_a,\pi_a) , \forall a \in K  \} 
\ ,
\end{equation}
this is clearly a $\rC^*$-algebra. 
Given the covariant representation (\ref{eq.Arep}), we also define the 
$\rC^*$-algebra 
\[
({\pi}^{\uparrow}_*,{\pi}^{\uparrow}_*)_U 
:= 
({\pi}^{\uparrow}_*,{\pi}^{\uparrow}_*) \cap 
\{  T \in \cB(\cH_o) : \ad U_{*,[p]}(T)=T \ , \ \forall p \in \pi_1^o(K)  \}
\ .
\]
It is clear that changing $o \in K$ we get a $\rC^*$-algebra isomorphic
to $({\pi}^{\uparrow}_*,{\pi}^{\uparrow}_*)_U$.

\begin{proposition}
\label{prop.red}
Let $(\cA,\jmath)_K$ be a net of $\rC^*$-algebras. Then for any representation 
$(\pi,U)$ of $(\cA,\jmath)_K$  there is an isomorphism
$({\pi}^{\uparrow}_*,{\pi}^{\uparrow}_*)_U   \simeq  \Si(\pi,U)$.  
\end{proposition}

\begin{proof}
We pick a path frame $P_o=\{ p_{(a,o)}, \  a\in K\}$ and define a map
\[
\beta : ({\pi}^{\uparrow}_*,{\pi}^{\uparrow}_*)_U   \to  \Si(\pi,U)
\ \ , \ \
\beta(w)_a :=  \ad U_{p_{(a,o)}}(w)  \ , \ \ a\in K 
\ .
\]
The above map has image in $\Si(\pi,U)$ since, for all $a \leq \tilde a$,
\begin{align*}
\beta(w)_{\tilde a}\,  U_{\tilde aa}  & = 
U_{p_{(\tilde a,o)}}\, w \, U_{p_{(o,\tilde a)}} \, U_{\tilde aa} = 
 U_{\tilde aa} \, \ad U_{\overline{(\tilde aa)}*p_{(\tilde a,o)}}(w)  \\
& = U_{\tilde aa} \, \left\{\ad U_{p_{(\tilde a,o)}}\circ 
\ad U_{*,[p_{(o,a)}*\overline{(\tilde aa)}*p_{(\tilde a,o)}]}\right\}(w) =
 U_{\tilde aa} \, \ad U_{p_{(\tilde a,o)}}(w) \\
& = U_{\tilde aa} \beta(w)_a \ ,
\end{align*}
%
%
and, for all $T \in \cA_a$, $a \in K$,
\begin{align*}
\beta(w)_a \, \pi_a(T) & =
\ad U_{p_{(a,o)}} ( w ) \, \pi_a(T)  =
\ad U_{p_{(a,o)}}(w) \, {\pi}^{\uparrow}_a (\iota_a,T) \\
&  =
\ad U_{p_{(a,o)}} ( w \, {\pi}^{\uparrow}_a(p_{(o,a)} , T ))  
= \ad U_{p_{(a,o)}}( {\pi}^{\uparrow}_a( p_{(o,a)} , T ) \, w)   \\ 
& = {\pi}^{\uparrow}_a( \iota_a , T ) \, \ad U_{p_{(a,o)}}(w)  =
\pi_a(T) \, \beta(w)_a
\ .
\end{align*}
Finally, $\beta$ is obviously an isometric $^*$-morphism, 
and defining
$\beta'(T) := T_o$,
$T \in \Si(\pi,U)$,
yields an inverse of $\beta$.
\end{proof}

\begin{definition}
A subrepresentation of $(\pi,U)$ is given by a family 
$\cH' := \{ \cH'_a \subseteq \cH_a \}_a$ 
of $\pi_a$-stable Hilbert subspaces such that
$U_{\tilde aa}\cH'_a = \cH'_{\tilde a}$, $\forall a \leq \tilde a$.
\end{definition}

\begin{corollary}
\label{cor.red}
Subrepresentations of $(\pi,U)$ are in one-to-one correspondence
with projections of $({\pi}^{\uparrow}_*,{\pi}^{\uparrow}_*)_U$.
\end{corollary}

\begin{proof}
By the previous proposition it suffices to  check 
projections of $\Si (\pi,U)$.
If $\cH'$ is a subrepresentation of $(\pi,U)$ then for each $a \in K$ 
we define $E_a \in B(\cH_a)$ to be  the projection of $\cH_a$ on $\cH'_a$. 
Clearly $E_a \in (\pi_a,\pi_a)$, and since 
$U_{\tilde aa} E_a v = U_{\tilde aa}v = E_{\tilde a}\, U_{\tilde aa}v$
for all $v \in \cH'_a$, we conclude that 
$\ad U_{\tilde aa}E_a = E_{\tilde a}$, $a \leq \tilde a$.
Conversely, if $E=E^2=E^* \in \Si (\pi,U)$ then we define $\cH'_a := E_a \cH_a$,
$a \in K$, and  check that $\cH'$ is a subrepresentation of $(\pi,U)$.
\end{proof}

Let $1_o \in \overline \cA_*$ denote the identity; clearly 
$1_o$ is invariant under the $\pi_1^o(K)$-action, thus,
by covariance, ${\pi}^{\uparrow}_*(1_o) \in ({\pi}^{\uparrow}_*,{\pi}^{\uparrow}_*)_U$.
This yields a subrepresentation of $(\pi,U)$ which, by construction, is unital
and whose complement  is a null representation.
We say that $(\pi,U)$ is \emph{non-degenerate} whenever ${\pi}^{\uparrow}_*(1_o)$
is the identity of $\cB(\cH_o)$; note that in this case each $\pi_o$, $o \in K$,
is a non-degenerate Hilbert space representation.

We say that $(\pi,U)$ is \emph{irreducible} whenever it does
not admit non-vanishing subrepresentations. By the previous corollary, the irreducibility  of $(\pi,U)$ is equivalent to the condition 
\[
({\pi}^{\uparrow}_*,{\pi}^{\uparrow}_*)_U \simeq \bC \ .
\]
Note that irreducibility of $(\pi,U)$ does not necessarily imply irreducibility
of the representations $\pi_a$, $a \in K$.


\

\noindent {\bf Cyclic vectors.} 
Let us consider, for each $a \in K$, the concrete $\rC^*$-algebra
generated by
\[
\{ 
\ad U_p \circ \pi_{\tilde a}(\cA_{\tilde a})
\ \ , \ \
p : \tilde a \to a
\}
\ \subseteq  \
\cB(\cH_a)
\]
coinciding, by Lemmas \ref{Cb:7} and \ref{Cb:8}, with ${\pi}^{\uparrow}_a (\overline \cA_a)$.
We fix $o \in K$ as usual, consider $v \in \cH_o$ and, given the path frame $P_o := \{ p_{(a,o)} \ , a\in K \}$, 
define the closed vector spaces
\[
U^av := \mathrm{closed \ span}\{ U_p U_{p_{(a,o)}} v \ , \ p : a \to a \}
\ \ , \ \
a \in K
\ .
\]
The space $U^av$ is independent of the choice of  path frame,  
since a different path frame $\tilde P_o=\{ \tilde p_{(a,o)} \ , a\in K\}$ yields 
$U_pU_{\tilde p_{(a,o)}}v = U_{ p *\tilde p_{(a,o)} *\overline{p_{(a,o)}}}\,  U_{p_{(a,o)}}v$.
We then define the family $\cH_v$ of Hilbert subspaces
\[
\cH_{v,a} := {\pi}^{\uparrow}_a (\overline \cA_a) U^a v
\ \ , \ \
a \in K
\ .
\]
Elementary computations show that 
\[
\pi_a(\cA_a)\cH_{v,a} \subseteq \cH_{v,a}
\ \ , \ \
U_{\tilde aa}\cH_{v,a} = \cH_{v,\tilde a}
\ \ , \ \
\forall a \leq \tilde a
\ ,
\]
proving the following result:
\begin{lemma}
Let $(\pi,U)$ be a representation of $(\cA,\jmath)_K$. Given $o \in K$ 
and $v \in \cH_o$, the pair $( \pi|_{\cH_v} , U|_{\cH_v} )$ yields a
subrepresentation of $(\pi,U)$.
\end{lemma}

\begin{definition}
Let $(\pi,U)$ be a representation over the family of Hilbert spaces $\cH$ 
and $o \in K$. Then $v \in \cH_o$ is said to be {\bf cyclic for $(\pi,U)$} 
whenever $\cH_{v,a} = \cH_a$ for any $a \in K$.
\end{definition}

\begin{proposition}
Let $(\pi,U)$ be a representation of the net $(\cA,\jmath)_K$.
Then cyclic vectors for $(\pi,U)$ are in one-to-one correspondence
with cyclic vectors for the induced crossed product
representation ${\pi}^{\uparrow}_* \rtimes U_* \ : \ 
\overline \cA_* \rtimes \pi_1^o(K) \ \to \ \cB(\cH_o)$. 
\end{proposition}

\begin{proof}
Let  $v \in \cH_o$ be cyclic for $(\pi,U)$. 
Since $p_{oo}$ is homotopic to the trivial loop we conclude that  
$U^ov = \{ U_pv , \ p : o \to o \}$,
thus 
$
\cH_{v,o} 
\ = \ 
\{ {\pi}^{\uparrow}_o(\overline \cA_o) \, U_p \, v \ , \ p : o \to o \}
$,
and the condition $\cH_{v,o} = \cH_o$ is clearly equivalent to the desired
cyclicity condition for ${\pi}^{\uparrow}_* \rtimes U_*$.
Conversely, assume that $v \in \cH_o$ is cyclic for ${\pi}^{\uparrow}_* \rtimes U_*$.
Then the above argument implies that
$\cH_o = \cH_{v,o}$,
and since each $U_p : \cH_{v,o} \to \cH_a$, $p : o \to a$, is unitary, we conclude
that $\cH_{v,o} = \cH_{v,a}$ for any $a \in K$, as desired.
\end{proof}

\begin{corollary}
Any non-degenerate representation $(\pi,U)$ is a direct sum of
cyclic representations.
\end{corollary}

\begin{proof}
It suffices to perform the direct sum decomposition of 
${\pi}^{\uparrow}_* \rtimes U_*$.
\end{proof}

\subsection{Vector states.}
\label{C.3}

We start with a preliminary remark on representations.
Let $(\cA,\jmath)_K$ be a net of $\rC^*$-algebras and 
$\pi := \{ \pi_a : \cA_a \to B(\cH_a) \}$
a family of unital representations on the family of Hilbert spaces $\cH = \{ \cH_a \}$.
We say that $(\pi,U)$ is a \emph{quasi-representation} whenever $U$ is a family of 
\emph{isometries} fulfilling
\[
U_{oa} \in ( \pi_a  , \pi_o \circ \jmath_{oa} )
\ \ , \ \
U_{eo} = U_{ea} \circ U_{ao}
\ \ , \ \
\forall a \leq o \leq e
\ .
\]
It is easily verified that the pair $(\cH,U)_K$ defines a net of Hilbert spaces.

\

\noindent {\bf The GNS construction.} 
Let $\omega$ be a state of the net of $\rC^*$-algebras $(\cA,\jmath)_K$.
Then the GNS construction yields a family $\cH := \{ \cH_a \}$ of Hilbert 
spaces with maps
$v_a : \cA_a \to \cH_a$, $a \in K$,
and a family of representations 
\[
\pi_a : \cA_a \to \cB(\cH_a)
\ \ , \ \
\{ \pi_a(T) \} \{ v_a(T') \} := v_a(TT')
\ \ , \ \
\forall T,T' \in \cA_a
\ ,
\]
having cyclic vectors $v_a := v(1_a) \in \cH_a$. 
Each $\jmath_{\tilde aa}$ defines the isometry 
\[
U_{\tilde aa} : \cH_a \to \cH_{\tilde a}
\ \ , \ \
U_{\tilde aa} v_a(T) := v_{\tilde a}( \jmath_{\tilde aa}T )
\ , \
T \in \cA_o
\ \ , \ \
a \leq \tilde a
\ ,
\]
and hence a quasi-representation $(\pi,U)$ of 
$(\cA,\jmath)_K$. Note that $v := \{ v_a \}$ is a section of $(\cH,U)_K$, 
i.e. $U_{\tilde aa}v_a = v_{\tilde a}$ for all $a \leq \tilde a$.

\

\noindent {\bf Vector states.} Let $(\cA,\jmath)_K$ be a net of 
$\rC^*$-algebras and $(\pi,U)$ a representation. A \emph{vector state} is
given by a family $v := \{ v_a \in \cH , \| v \| = 1 \}$ such that 
$v_{\tilde a} = U_{\tilde aa} v_a$, $a \leq \tilde a$.
This induces the state $\omega_a (\cdot) := (v_a , \cdot v_a)$,
$a \in K$.

A representation does not necessarily yield vector states, indeed
this happens if and only if the underlying Hilbert net bundle $\cH$ 
has sections. In the following  we give an example of net of 
$\rC^*$-algebras having representations but no (vector) states.

\begin{example}
\label{ex_free}
Let $\Gamma_n$ denote the free group with $n$ generators and $\cA_* := C_b(\Gamma_n)$ 
the $\rC^*$-algebra of bounded continuous functions w.r.t. the Haar measure.
We denote the left translation  by
\begin{equation}
\label{La}
\lambda : \Gamma_n \times \cA_* \to \cA_*
\ \ , \ \
g,A \mapsto \lambda_g A =: A_g
\ .
\end{equation}
Since $\Gamma_n$ is not amenable, there are no $\Gamma_n$-invariants states of $\cA_*$.
Now let  $M$ be a space with $\pi_1(M) = \Gamma_n$ (for example, the $n$-{\em bouquet})
and $K$ a good base for the topology of $M$. Fix $o\in K$, so that 
$\pi^o_1(K) \simeq \Gamma_n$,  
and define $(\cA_{**},\lambda_*)_K$ as the $\rC^*$-net bundle  associated to the 
dynamical system $(\cA_*,\pi^o_1(K),\lambda)$ (\ref{Bc:8}). 
%
By Lemma \ref{Ca:4} we conclude that $(\cA_{**},\lambda_*)_K$ does not have states.
%
%
\end{example}

\noindent {\bf Projective states.} Let $(\cA,\jmath)_K$ be a 
net of $\rC^*$-algebras and $(\pi,U)$ a representation over the Hilbert net bundle
$(\cH,U)_K$. Assume that there is a subnet of Hilbert spaces $(\cL,\lambda)$ of
$(\cH,U)_K$ with rank $1$. This means that there is a family 
$v = \{ v_a \in \cH_a \}$ 
of normalized vectors generating the subspaces $\cL_a \subseteq \cH_a$, $a \in K$, 
such that, with
$\lambda = \{ \lambda_{\tilde aa} \in \bT  , a \leq \tilde a \}$,
we have
\[
U_{\tilde aa} v_a = \lambda_{\tilde aa} v_{\tilde a}
\ \ , \ \
a \leq \tilde a
\ .
\]
In this scenario, the state $\omega \in \cS(\cA,\jmath)_K$, 
$\omega_a := (v_a , \cdot v_a)$, $a \in K$,
is well defined, and we say $\omega$ is a {\em projective state of $(\cA,\jmath)_K$
defined in $(\cH,U)_K$}.
The following result is an immediate consequence of the remarks at the beginning
of this appendix.

\begin{proposition}
Let $(\cA,\jmath)_K$ a net of $\rC^*$-algebras and $(\pi,U)$ a representation
over the family of Hilbert spaces $\cH$.
Then projective states defined in $(\cH,U)_K$ are in one-to-one correspondence
with rank one projections of $(U_*,U_*) \subseteq \cB(\cH_o)$.
\end{proposition}

\noindent {\small {\bf Aknowledgements.} 
We gratefully acknowledge the hospitality
and support  of  the Graduate School of 
Mathematical Sciences of the University of Tokyo, where part of this paper has been developed, 
in particular Yasuyuki Kawahigashi for his warm hospitality. We also would like to thank 
\emph{all} the operator algebra group of the University of Roma ``Tor Vergata'', 
Sebastiano Carpi and Fabio Ciolli,
for the several fruitful discussions on the topics treated in  this paper.}



\begin{thebibliography}{99}
\markboth{Bibliography}{Bibliography}


\bibitem{Ara}
H. Araki. {\em Mathematical theory of quantum fields.}  Oxford University Press, Oxford, 2009.


\bibitem{AM}
H. Araki, H. Moriya: Equilibrium statistical mechanics of fermion lattice systems. 
Rev. Math. Phys. {\bf 15} (2003), no. 2, 93--198.
Available as \texttt{arXiv:math-ph/0211016v2} 


\bibitem{BrF}
R. Brunetti, K. Fredenhagen:
Algebraic approach to Quantum Field Theory.
To appear on Elsevier Encyclopedia of Mathematical Physics.
Available as  \texttt{arXiv:math-ph/0411072v1}. 

\bibitem{BF} 
    D. Buchholz, K. Fredenhagen: Locality and the structure of particle states. 
    Commun. Math. Phys. {\bf 84}, (1982), 1--54.



\bibitem{BFM}
R. Brunetti, L. Franecschini, V. Moretti: 
Topological features of massive bosons on two dimensional Einstein space-time.
Ann. Henri Poincar\'e {\bf 10} (2009), no. 6, 1027--1073. 
Available as  \texttt{arXiv:0812.0533v3 [gr-qc]}  



\bibitem{BGP}
C. B\"ar, N. Ginoux, F. Pf\"affle.
{\em Wave equations on Lorentzian manifolds and quantization.} 
ESI Lectures in Mathematics and Physics. European Mathematical Society (EMS), Z\"urich, 2007.
Available as \texttt{arXiv:0806.1036v1 [math.DG]}

\bibitem{BH}
D. Buchholz, R. Haag: The Quest for Understanding in Relativistic Quantum Physics 
J. Math. Phys. {\bf 41} (2000), no. 6, 3674--3697. 
Available as \texttt{arXiv:hep-th/9910243v2}

\bibitem{BHa}
M.R. Bridson, A. Haefliger: 
{\em Metric Spaces of Non-Positive Curvature.} Springer-Verlag, Berlin (1999)

\bibitem{Bil}
T.S. Bildea: Generalized free amalgamated product of $\rC^*$-algebras. 
J. Operator Theory {\bf 56} (2006), no. 1, 143--165. 
Available as \texttt{arXiv:math/0408434v2 [math.OA]}

\bibitem{Bla}
B.E. Blackadar: 
Weak expectations and nuclear $\rC^*$-algebras. 
Indiana Univ. Math. J. {\bf 27} (1978), no. 6, 1021--1026. 





\bibitem{BMT}
D. Buchholz, G. Mack, I. Todorov: The current algebra on the circle as a germ of local field theories. Conformal field theories and related topics. Nuclear Phys. B Proc. Suppl. {\bf 5B} (1988), 20--56.

\bibitem{BR} 
R. Brunetti, G. Ruzzi: 
Quantum charges and spacetime topology: The emergence of new superselection sectors. 
Comm. Math. Phys. {\bf 287} (2009), no. 2, 523--563. 
Available as  \texttt{arXiv:0801.3365v2 [math-ph]} 

\bibitem{BrR}
O. Bratteli, D.W. Robinson.\emph{Operator algebras and quantum statistical mechanics. 2. Equilibrium states. Models in quantum statistical mechanics.} Second edition. Texts and Monographs in Physics. Springer-Verlag, Berlin, 1997. 

\bibitem{BFV}
R. Brunetti, K. Fredenhagen, R. Verch:
The generally covariant locality principle - A new paradigm
for local quantum physics. Commun. Math. Phys. {\bf 237}, (2003), 31--68. 
Available as  \texttt{arXiv:math-ph/0112041v1}.

\bibitem{CKL}
 S. Carpi, Y. Kawahigashi, R. Longo:
  Structure and classification of superconformal nets.
  Ann. Henri Poincar\'e  {\bf 9} (2008), no. 6, 1069--1121. 
 Available as \texttt{arXiv:0705.3609v2 [math-ph]}. 

\bibitem{CRV}
F. Ciolli, G. Ruzzi, E. Vasselli, work in progress. 

\bibitem{DFK}
C. D'Antoni, K. Fredenhagen, S. K\"oster:
Implementation of conformal covariance by diffeomorphism symmetry.  
Lett. Math. Phys. {\bf 67} (2004), no. 3, 239--247. 
Available as \texttt{arXiv:math-ph/0312017v1}.

\bibitem{DHR1}
   S. Doplicher, R. Haag, J. E. Roberts:
   Local observables and particle statistics I.
   Commun. Math Phys. {\bf 23}, (1971), 199--230.

\bibitem{DHR2}  
   S. Doplicher, R. Haag, J. E. Roberts:
   Local observables and particle statistics II.
   Commun. Math Phys. {\bf 35}, (1974), 49--85 .

\bibitem{Dim}
J. Dimock:  Algebras of local observables on a manifold. 
Comm. Math. Phys. {\bf 77} (1980), no. 3, 219--228.


\bibitem{FH}
K. Fredenhagen, R. Haag: Generally covariant quantum field theory and scaling limits. 
Comm. Math. Phys. {\bf 108} (1987), no. 1, 91--115. 

\bibitem{Fol}
G. Folland. 
{\em A course in abstract harmonic analysis.} 
Studies in Advanced Mathematics. CRC Press, Boca Raton, FL, 1995. 


\bibitem{Fre}
K. Fredenhagen: 
Generalizations of the theory of superselection sectors. 
The algebraic theory of superselection sectors (Palermo, 1989), 379--387, World Sci. Publ., River Edge, NJ, 1990.


\bibitem{FRS}
 K. Fredenhagen, K.-H. Rehren, B. Schroer: Superselection sectors with braid group statistics and exchange algebras. II. Geometric aspects and conformal covariance. Special issue dedicated to R. Haag on the occasion of his 70th birthday. Rev. Math. Phys. 1992, Special Issue, 113--157.
  
\bibitem{GHV} 
Group theory from a geometrical viewpoint. 
Proceedings of the workshop held in Trieste, March 26–April 6, 1990. 
Editors  \'E. Ghys, A. Haefliger, A. Verjovsky. World Scientific Publishing Co., Inc., River Edge, NJ, (1991).  
  
  
\bibitem{GL}
D. Guido, R. Longo: Relativistic invariance and charge 
conjugation in quantum field theory. 
Comm. Math. Phys. {\bf 148} (1992), no. 3, 521--551. 

\bibitem{GLRV}
D. Guido, R. Longo, J.E. Roberts, R. Verch:
Charged sectors, spin and statistics in quantum field theory on curved spacetimes.  
Rev. Math. Phys. {\bf 13} (2001), no. 2, 125--198. 
Available as  \texttt{arXiv:math-ph:9906019v1}.

\bibitem{Haa}
R. Haag. {\em Local Quantum Physics.}
Springer Texts and Monographs in Physics,  1996,  2nd edition.

\bibitem{Hae}
A. Haefliger:  Complexes of groups and orbihedra. Group theory from a geometrical viewpoint. (Trieste, 1990), 504–-540, World Sci. Publ., River Edge, NJ, 1991. 





\bibitem{HK}
R. Haag, D. Kastler:
An algebraic approach to quantum field theory.
J. Math. Phys. {\bf 43}, 848--861 (1964).

\bibitem{Hol}
S. Hollands: 
Renormalized quantum Yang-Mills fields in curved spacetime.  
Rev. Math. Phys. {\bf 20} (2008), no. 9, 1033--1172. 
Available as \texttt{arXiv:0705.3340v3 [gr-qc]}.


\bibitem{Neu1}
H. Neumann: Generalized free products with amalgamated subgroups. Amer. J. Math. {\bf 70}, (1948). 590--625. 

\bibitem{Neu2}
H. Neumann:
Generalized free products with amalgamated subgroups. II. 
Amer. J. Math. {\bf 71}, (1949),  491--540. 



\bibitem{Rob1}
J.E. Roberts: 
Local cohomology and superselection structure. 
Comm. Math. Phys. {\bf 51} (1976), no. 2, 107--119. 

\bibitem{Rob}
J. E. Roberts: Lectures on algebraic quantum field theory.  The algebraic theory of superselection sectors. (Palermo, 1989), 1--112, World Sci. Publ., River Edge, NJ, 1990. 

\bibitem{RR}
   J. E. Roberts, G. Ruzzi:
   A cohomological description of connections and curvature over posets.
   Theo. App. Cat.  \textbf{16}  (2006), No.30, 855--895.
   Available as  \texttt{arXiv:math/0604173v1 [math.AT]}.

\bibitem{RRV}  
 J.E.Roberts, G. Ruzzi, E. Vasselli: A theory of bundles over posets. 
 Adv. Math. {\bf 220} (2009), no. 1, 125--153. 
 Available as \texttt{arXiv:0707.0240v1 [math.AT]}.

\bibitem{RS} 
K.-H. Rehren, B. Schroer:  Quasiprimary fields: an approach to positivity of $2$D conformal quantum field theory. Nuclear Phys. B {\bf 295} (1988), no. 2, FS21, 229--242. 

\bibitem{Ruz}
G. Ruzzi: Homotopy of posets, net-cohomology, and theory of superselection sectors in globally hyperbolic  spacetimes.  Rev. Math. Phys. {\bf 17} (9), 1021--1070 (2005). 
Available as \texttt{arXiv:math-ph/0412014v2}. 

\bibitem{RV2}
G. Ruzzi, E. Vasselli: The $C_0(X)$-algebra of a net and index theory, in preparation.


\bibitem{RVs1}
G. Ruzzi, E. Vasselli: Representations of nets of $C^*$-algebras over $S^1$, in preparation.




\bibitem{Sta} 
J.R. Stallings: Non-positively curved triangles of groups.
Group theory from a geometrical viewpoint (Trieste, 1990), 491--503, 
World Sci. Publ., River Edge, NJ, 1991. 



\bibitem{Ve}
R. Verch: Continuity of symplectically adjoint maps and the algebraic structure
of  Hadamard vacuum representations for quantum fields in
curved spacetime.
Rev. Math. Phys. {\bf 9}, No.5, (1997), 635--674.
Available as \texttt{arXiv:funct-an/9609004v2}.





\end{thebibliography}
\end{document}